\documentclass[reqno, final,a4paper]{amsart}

\usepackage{color}
\usepackage[colorlinks=true,allcolors=blue,backref=page]{hyperref}
\usepackage{amsmath, amssymb, amsthm}
\usepackage{mathrsfs}
\usepackage{mathtools}
\usepackage[noabbrev,capitalize,nameinlink]{cleveref}

\crefname{equation}{}{}
\AtBeginEnvironment{appendices}{\crefalias{section}{appendix}} %appendices
\usepackage{fullpage}
\usepackage[noadjust]{cite}
\usepackage{graphics}
\usepackage{pifont}
\usepackage{tikz}
\usepackage{bbm}
\usepackage[T1]{fontenc}
\usetikzlibrary{arrows.meta}

\usepackage{environ}
\usepackage{framed}
\usepackage{url}
\usepackage[linesnumbered,ruled,vlined]{algorithm2e}
\usepackage[noend]{algpseudocode}
\usepackage[labelfont=bf]{caption}
\usepackage{framed}
\usepackage[framemethod=tikz]{mdframed}
\usepackage{appendix}
\usepackage{graphicx}
\usepackage[textsize=tiny]{todonotes}
\usepackage{tcolorbox}
\usepackage{pdfpages}
\allowdisplaybreaks

\let\originalleft\left
\let\originalright\right
\renewcommand{\left}{\mathopen{}\mathclose\bgroup\originalleft}
\renewcommand{\right}{\aftergroup\egroup\originalright}

\crefname{algocf}{Algorithm}{Algorithms}

\crefname{equation}{}{} %remove ``Equation''
\usepackage[color,final]{showkeys} %add in 'final' into parameter to remove showkeys

\colorlet{refkey}{orange!20}
\colorlet{labelkey}{blue!30}

\crefname{algocf}{Algorithm}{Algorithms}

\numberwithin{equation}{section}
\newtheorem{theorem}{Theorem}[section]

\newtheorem{lemma}[theorem]{Lemma}
\newtheorem{claim}[theorem]{Claim}
\crefname{claim}{Claim}{Claims}

\newtheorem*{question*}{Question}
\newtheorem{fact}[theorem]{Fact}

\theoremstyle{definition}
\newtheorem{definition}[theorem]{Definition}

\newtheorem*{definition*}{Definition}

\theoremstyle{remark}
\newtheorem{remark}[theorem]{Remark}

\newcommand{\floor}[1]{\left\lfloor #1 \right\rfloor}

\newcommand{\mb}{\mathbb}
\newcommand{\mbf}{\mathbf}
\newcommand{\mbm}{\mathbbm}
\newcommand{\mc}{\mathcal}

\newcommand{\mr}{\mathrm}

\newcommand{\on}{\operatorname}

\allowdisplaybreaks

\title{A central limit theorem for the matching number of a sparse random graph}

\author{Margalit Glasgow}
\address{Department of Computer Science, Stanford University, Stanford, CA.}
\email{mglasgow@stanford.edu}
\author{Matthew Kwan}
\address{Institute of Science and Technology Austria (ISTA).}
\email{matthew.kwan@ist.ac.at}
\author{Ashwin Sah}
\address{Department of Mathematics, Massachusetts Institute of Technology, Cambridge, MA 02139, USA}
\email{\{asah,msawhney\}@mit.edu}
\author{Mehtaab Sawhney}
\thanks{
Kwan was supported by ERC Starting Grant ``RANDSTRUCT'' No.~101076777.
Sah and Sawhney were supported by NSF Graduate Research Fellowship Program DGE-2141064. Sah was supported by the PD Soros Fellowship.
}

\begin{document}
\begin{abstract}
In 1981, Karp and Sipser proved a law of large numbers for the matching
number of a sparse Erd\H os--R\'enyi random graph, in an influential
paper pioneering the so-called \emph{differential equation method}
for analysis of random graph processes. Strengthening this classical
result, and answering a question of Aronson, Frieze and Pittel, we
prove a central limit theorem in the same setting: the fluctuations
in the matching number of a sparse random graph are asymptotically
Gaussian.

Our new contribution is to prove this central limit theorem in the \emph{subcritical} and \emph{critical} regimes, according to a celebrated algorithmic phase transition first observed by Karp and Sipser. Indeed, in the \emph{supercritical} regime, a central limit theorem has recently been proved in the PhD thesis of Krea\v ci\'c, using a stochastic generalisation of the differential equation method (comparing the so-called \emph{Karp--Sipser process} to a system of stochastic differential equations). Our proof builds on these methods, and introduces new techniques to handle certain degeneracies present in the subcritical and critical cases. Curiously, our new techniques lead to a \emph{non-constructive} result: we are able to characterise the fluctuations of the matching number around its mean, despite these fluctuations being much smaller than the error terms in our best estimates of the mean.

We also prove a central limit theorem for
the rank of the adjacency matrix of a sparse random graph.
\end{abstract}

\maketitle

\section{Introduction}\label{sec:introduction}

\makeatletter
\newcommand*{\transpose}{%
  {\mathpalette\@transpose{}}%
}
\newcommand*{\@transpose}[2]{%
  % #1: math style
  % #2: unused
  \raisebox{\depth}{$\m@th#1\intercal$}%
}
\makeatother

One of the foundational theorems in random graph theory, proved by
Erd\H os and R\'enyi~\cite{ER66} in 1966, characterises the asymptotic
probability that a random graph contains a \emph{perfect matching}
(i.e., that we can pair up all the vertices of the graph using disjoint
edges). In particular, this property has a \emph{sharp threshold}:
for any positive constant $\varepsilon>0$, and a large even integer
$n$, random graphs with $n$ vertices and more than $((1+\varepsilon)\log n)\cdot n/2$
edges are very likely to contain perfect matchings, while random graphs
with $n$ vertices and fewer than $((1-\varepsilon)\log n)\cdot n/2$
edges are very likely \emph{not} to contain perfect matchings.

Below the perfect matching threshold, it is typically not possible
to pair up \emph{all }the vertices, but it is still natural to ask
what fraction of vertices can be paired up. In 1981, Karp and Sipser~\cite{KS81}
provided an asymptotic answer to this question, as follows.
\begin{theorem}
\label{thm:KS}Fix a constant $c>0$, consider a set of $n$
vertices, and let $G$ be a random graph defined in one of the following
two ways\footnote{The name ``Erd\H os--R\'enyi random graph'' is used to refer
to both these notions of a random graph: either we independently include
each edge with a given probability, or we choose a uniformly random
graph with a given number of edges. These two models are closely related,
and the same types of techniques can be used to study both. Here we
have chosen the parameters in such a way that, in both models, the
average degree is likely to be about $c$.}:
\begin{itemize}
\item $G$ contains each of the $\binom{n}{2}$ possible edges with probability
$c/n$ independently, or 
\item $G$ contains a uniformly random subset of exactly $\floor{c n/2}$
of the possible edges.
\end{itemize}
Let $\nu(G)$ be the matching number of $G$ (i.e., the maximum size
of a set of disjoint edges in $G$). Then, for some constant $\alpha_{c}\in[0,1]$,
we have the convergence in probability
\[
\frac{\nu(G)}{n/2}\overset{p}{\to}\alpha_{c}
\]
as $n\to\infty$. Specifically, $\alpha_{c}=\min_{x\in[0,1]}\big(2-\exp(-c\exp(-c(1-x)))-(1+c(1-x))\exp(-c(1-x))\big)$.
\end{theorem}

In their proof of \cref{thm:KS}, Karp and Sipser introduced a number
of highly influential ideas. First, they introduced a random graph
process (now often called the \emph{Karp--Sipser leaf-removal process}; we define it in \cref{def:KS})
designed to construct a near-maximum matching in a random graph. This
process has since found a number of important applications outside
the context it was originally introduced (e.g., in statistical physics,
theoretical computer science and random matrix theory~\cite{BG01,BG01b,MRZ03,BLS11,CCKLRb,GKSS}).
In order to analyse the behaviour of this process, Karp and Sipser
identified certain statistics (evolving with the process), such that
the random trajectories of these statistics concentrate around a deterministic
``limit trajectory'', described as the solution to a certain system
of differential equations. This is arguably the first application
of the so-called \emph{differential equation method} for random graph
processes (see the surveys in \cite{Wor99,DN08}), which has had an
enormous impact in combinatorics and theoretical computer science.

Note that \cref{thm:KS} can be interpreted as a \emph{law of large
numbers}: with high probability, the matching number $\nu(G)$ is
close to its expected value $\mb E\nu(G)=\alpha_{c}(n/2)+o(n)$.
From this point of view, it is natural to wonder whether there is
a corresponding \emph{central limit theorem} in the same setting:
are the fluctuations of $\nu(G)$ around its mean $\mb E\nu(G)$ asymptotically
Gaussian? This question seems to have been first explicitly asked
in a 1998 paper of Aronson, Frieze and Pittel~\cite{AFP98}. As our
main result, we answer this question, proving a central limit theorem
for the matching number.
\begin{theorem}
\label{thm:main}Define $G$ and $\nu(G)$ as in \cref{thm:KS}.
Then we have the convergence in distribution
\[
\frac{\nu(G)-\mb E\nu(G)}{\sqrt{\on{Var}\nu(G)}}\overset{d}{\to}\mathcal{N}(0,1),
\]
as $n\to\infty$. Moreover, the asymptotics of $\operatorname{Var}\nu(G)$ (which is of order $n$)
can be explicitly described in terms of an integral involving the
solution to a certain system of differential equations; see \cref{rem:limiting-variance}.
\end{theorem}

We remark that in the case $c<1$, the statement of \cref{thm:main}
follows from powerful general results of Pittel~\cite{Pit90} (when $c<1$, random graphs have a very simple structure with no large connected components, and this structure
can be very precisely characterised). However, we will not need this in our proof.

More significantly, in the case $c>e$, the statement of \cref{thm:main} was recently proved in the PhD thesis of Krea\v ci\'c~\cite{Kre17} via a beautiful \emph{stochastic} generalisation of the differential equation method. This work considers the same statistics of the Karp--Sipser leaf-removal
process that were considered in Karp and Sipser's seminal paper, but
instead of simply proving that the trajectory of these statistics
concentrates around a deterministic limit, it proves that this trajectory
converges in distribution to a certain Gaussian process, obtained
as the solution to a system of \emph{stochastic} differential equations.
This is accomplished via a general limit theorem for Markov chains
due to Ethier and Kurtz~\cite{EK86}. (See also the related techniques of Janson and Luczak~\cite{JL08}, using a martingale limit theorem by Jacod and
Shiryaev~\cite{JS03} to prove central limit theorems for the so-called \emph{$k$-core problem}).

Unfortunately, this type of analysis breaks down when $c\le e$, due to the notorious \emph{phase transition} of the Karp--Sipser process: when $c>e$, the process remains ``macroscopic'' until its termination, whereas when $c\le e$ the process becomes more and more degenerate as it reaches completion (and one loses control over all relevant statistics). For the purposes of a law of large numbers (\cref{thm:KS}) this short period of degenerate behaviour can be ignored (its contribution is trivially $o(n)$), but for the purposes of a central limit theorem (\cref{thm:main}) it is not clear how to rule out dangerously large fluctuations during this short degenerate period\footnote{It is difficult to precisely describe the issue without rather a lot of setup; a concrete description of the relevant problem will eventually appear at the end of \cref{subsec:general-CLT}.}.

A number of additional ideas are therefore required. In particular, we show how to combine Gaussian process approximation with coupling and concentration inequalities (and a careful stability analysis of a certain system of differential equations) to prove our central limit theorem
\emph{non-constructively}: we are able to prove a central limit theorem
for $\nu(G)$ around its mean $\mb E\nu(G)$, despite not having any way to actually determine
the value of $\mb E\nu(G)$ (of course, we have the estimate $\mb E\nu(G)=\alpha_c n+o(n)$
from \cref{thm:KS}, but the error term here is much larger than the
typical fluctuations of $\nu(G)$).
\begin{remark}
There are a number of powerful general techniques to prove central
limit theorems in random graphs (see for example \cite[Section~6]{JLR00}). In particular, there is a recent
general framework due to Cao~\cite{Cao21} which is suitable for
proving central limit theorems for a broad range of graph parameters defined
in terms of optimisation problems with a ``long-range
independence'' property. While a certain form of this property is satisfied
for the maximum matching problem, it is not satisfied in a strong
enough way\footnote{We remark that in the ``smoother'' setting of \emph{weighted} sparse random graphs (in which a random $\on{Exponential}(1)$ weight is assigned to each edge), it was proved by Gamarnik, Nowicki and Swirszcz~\cite{GNS06} that the necessary long-range independence property \emph{is} satisfied, so in this setting a central limit theorem (for the maximum weight of a matching) does immediately follow.} to apply the techniques in \cite{Cao21}.
\end{remark}

\subsection{The rank of a random graph}

Let $\operatorname{rk}(G)$ be the rank of the adjacency matrix of
a graph $G$. It turns out that $\on{rk}(G)$ is very closely related
to $\nu(G)$, due to a connection between both of
these parameters and the Karp--Sipser leaf-removal process. It was
first observed by Bordenave, Lelarge and Salez~\cite{BLS11} that
the statement of \cref{thm:KS} holds with $\operatorname{rk}(G)/2$
in place of $\nu(G)$ (a more general connection between $\operatorname{rk}(G)$
and $\nu(G)$ was subsequently conjectured by Lelarge~\cite{Lel13}
and proved by Coja-Oghlan, Erg\"ur, Gao, Hetterich, and Rolvien~\cite{CEGHR20}).
Using similar techniques as for \cref{thm:main}, we
are able to prove a central limit theorem for $\operatorname{rk}(G)$.
\begin{theorem}\label{thm:rank-CLT}
The statement of \cref{thm:main} holds with $\operatorname{rk}(G)$
in place of $\nu(G)$.
\end{theorem}

We remark that we recently proved the $c>e$ case of \cref{thm:rank-CLT} in \cite{GKSS} as a \emph{corollary} of the $c>e$ case of \cref{thm:main}, using a combinatorial description of the rank of a sparse random graph (which was the main result of \cite{GKSS}). Again, our main contribution here is the case $c\le e$.

\subsection{Notation}\label{sub:notation}
We use standard asymptotic notation throughout, as follows. For functions $f=f(n)$ and $g=g(n)$, we write $f=O(g)$ to mean that there is a constant $C$ such that $|f(n)|\le C|g(n)|$ for sufficiently large $n$. Similarly, we write $f=\Omega(g)$ to mean that there is a constant $c>0$ such that $f(n)\ge c|g(n)|$ for sufficiently large $n$. We write $f=\Theta(g)$ to mean that $f=O(g)$ and $g=\Omega(f)$, and we write $f=o(g)$ to mean that $f(n)/g(n)\to0$ as $n\to\infty$. Subscripts on asymptotic notation indicate quantities that should be treated as constants.

\subsection{Acknowledgments}
We would like to thank Christina Goldschmidt and Eleonora Krea\v ci\'c for insightful discussions and clarifications about their work in the thesis \cite{Kre17}.

\section{Proof ideas}\label{sec:ideas}
In this section we provide a high-level sketch of some of the main ideas behind the proofs of \cref{thm:main,thm:rank-CLT}, along the way introducing some key definitions and results from the literature.

\subsection{Karp--Sipser leaf removal} First, we define the Karp--Sipser leaf removal algorithm.
\begin{definition}[Karp--Sipser leaf removal]\label{def:KS}
Starting from a (multi)graph $G$, repeatedly do the following. As
long as there exist degree-1 vertices (\emph{leaves}), choose one uniformly at random,
and delete both this vertex and its unique neighbour. Let $G(i)$
be the graph remaining after $i$ steps of this process, minus its isolated vertices (so $G(0)$ consists of $G$ without its isolated vertices).
\end{definition}

It is easy to see that a single step of leaf-removal decreases the
matching number of $G$ by exactly one, and the rank of $G$ by exactly
two. So, for any $G$ and any time $i$ (for which $G(i)$ is defined),
writing $\alpha(G)=\nu(G)$ or $\alpha(G)=\on{rk}(G)/2$, we always
have
\begin{equation}
\alpha(G)=i+\alpha(G(i)).\label{eq:KS-bound}
\end{equation}
It is natural to continue the leaf-removal process as long as possible, until the
point when we run out of leaves (let $I$ be this point in time,
so $G(I)$ may or may not be empty, but definitely has no leaves). If $G$ is an Erd\H os--R\'enyi random
graph, then the final ``Karp--Sipser core'' $G(I)$ is quite
well-behaved: the distribution of $G(I)$ has an explicit description
in terms of some simple statistics of $G(I)$, and one can study
its rank or matching number $\alpha(G(I))$ directly. So, in order
to prove a central limit theorem for $\alpha(G)$, a sensible strategy is to study the joint distribution of $I$ and of certain statistics
of $G(I)$, and then study $\alpha(G(I))$ in terms of these statistics and apply \cref{eq:KS-bound}. This was precisely the approach taken in \cite{Kre17,GKSS}
for the regime $c>e$.

The significance of the distinction between $c\le e$ and $c>e$ is
that it represents a ``phase transition'' for the behaviour of the above leaf--removal process up to time $I$, as follows.
\begin{theorem}
\label{thm:phase-transition}Fix a constant $c>0$, and let $G$ be
as in \cref{thm:KS}. Run the Karp--Sipser leaf-removal process
until the time $I$ when $G(I)$ has no leaves remaining; then the number of vertices $v(G(I))$ in $G(I)$ satisfies
\[
\frac{v(G(I))}{n}\overset{p}{\to}\beta_{c},
\]
where $\beta_{c}=0$ for $c\le e$ and $\beta_{c}>0$ for $c>e$.
\end{theorem}
That is to say, if $c>e$ then the Karp--Sipser core $G(I)$ has size comparable to
$G$, whereas if $c\le e$ then $G(I)$ is vanishingly small compared to $G$. This phase transition was first observed by Karp and Sipser~\cite{KS81}, and is now sometimes called the ``$e$-phenomenon''. % (somewhat surprisingly, this is closely connected to certain other spectral phase transitions  appearing in a variety of other contexts~\cite{CCKLRb,CS21})
A number of alternative proofs are now available (see for example \cite{CLR,AFP98,JS21}, and the physics-based heuristics in \cite{ZM06}).

Actually, in later work by Aronson, Frieze and Pittel~\cite{AFP98}, in the \emph{strictly} subcritical regime $c<e$ it was shown that $G(I)$ is truly tiny: its expected number of vertices is only $O_c(1)$. From a certain point of view, this makes the strictly subcritical regime seem \emph{easier} than the supercritical regime $c>e$. Indeed, in the strictly subcritical regime, $\alpha(G(I))$ is trivially almost zero, so recalling \cref{eq:KS-bound} it suffices to prove a central limit theorem for $I$. %\mg{Am I missing something? You still need to understand the something about $\alpha(G(I))$ since it only goes to $0$ when divided by $n$?  What about: "Indeed, in the subcritical regime, $\alpha(G(I)) = o(n)$, so we would expect that its fluctuations would be on a smaller scale than the fluctuations in $I$. Thus recalling \cref{eq:KS-bound}, it would suffice to prove a central limit theorem for $I$." } 
On the other hand, if $c>e$, one must work hard\footnote{It turns out that $\alpha(G(I))$ is very likely to be very nearly equal to half the number of vertices in the Karp--Sipser core $G(I)$. Given this approximation, for the purpose of proving a central limit theorem, one only needs to understand the joint distribution of $I$ and the number of vertices in the Karp--Sipser core. However, it is a highly nontrivial matter to actually prove this approximation, especially in the case $\alpha=\on{rk}/2$; see \cite{AFP98,FP04,GKSS}.} to understand $\alpha(G(I))$, in addition to studying fluctuations related to the Karp--Sipser process\footnote{One may also wonder about the Karp--Sipser core $G(I)$ in the \emph{critical} regime $c=e$. This regime is much more difficult to understand, and beyond \cref{thm:phase-transition} essentially nothing has been rigorously proved (though see the very recent work of Budzinski, Contat, and Curien~\cite{BCC22} on a simpler model of random graphs, and the numerical simulations of Bauer and Golinelli~\cite{BG01b}).}.

However, in the regime $c\le e$ the leaf-removal process becomes more and more degenerate as we reach the end of the process. As $G(i)$ becomes smaller and smaller, we begin to lose law-of-large-numbers-type effects, and it is very difficult to maintain control over the evolution (and fluctuation) of various statistics\footnote{Concretely, if one attempts to proceed as in \cite{Kre17} in the case $c\le e$, one runs into an indeterminate division 0/0; this is discussed further in \cref{subsec:general-CLT}.}. This is the key reason that the $c\le e$ cases of \cref{thm:main,thm:rank-CLT} were open until now.

\subsection{A stopped central limit theorem}\label{subsec:stop-idea}

In the regime $c\le e$ of interest, our approach is instead to stop the process \emph{well before} time
$I$, while it is ``not too degenerate''. Indeed, let $G\sim \mb G(n,c/n)$ be an Erd\H os--R\'enyi random graph, and fix some $\delta>0$ (which we view, for now, as a constant not depending on $n$). Consider the leaf-removal process as defined in \cref{def:KS}, and let $I_\delta$ be the first time $i$ at which $G(i)$ has at most $\delta n$ edges. Then, it is straightforward to adapt the techniques
in \cite{Kre17} to prove a central limit theorem for $I_\delta$. (A precise statement of this central limit theorem appears as \cref{thm:KS-CLT}, and in \cref{sec:CLT} we provide a sketch of the proof strategy in \cite{Kre17}, which is interesting in its own right).

Recall from \cref{eq:KS-bound} that $\alpha(G)=I_{\delta}+\alpha(G(I_{\delta}))$. So, given a central limit theorem for $I_{\delta}$, in order to prove
a central limit theorem for $\alpha(G)$, it suffices to show that
the fluctuations in $\alpha(G(I_{\delta}))$ are small compared
to the fluctuations in $I_{\delta}$ (specifically, it suffices to show this is true in the limit $\delta\to 0$; indeed, although our central limit theorem is stated for constant $\delta>0$, a compactness argument shows that it also holds if $\delta\to0$ sufficiently slowly).

\begin{remark}
The above strategy only makes sense for $c\le e$. Indeed, recall from
\cref{thm:phase-transition} that in the regime $c>e$, the ``core''
$G(I)$ is large, so for any reasonable stopping time $I'$ one should expect $\alpha(G(I'))$
to have rather large fluctuations.
\end{remark}

At a very high level, this describes our strategy to prove \cref{thm:main,thm:rank-CLT}. It is not an easy matter to actually prove (for $c\le e$ and $\delta\to 0$) that the fluctuations in $I_{\delta}$ dominate the fluctuations in $\alpha(G(I_{\delta}))$; indeed, this task occupies almost all of the paper. However, we believe that it is very intuitive, morally speaking, that this should be the case. Indeed, in the case $c\le e$, imagine running the process \emph{backwards} from $I=I_0$; at time $I$ the ``Karp--Sipser core'' is basically empty and has tiny fluctuations, while we should expect plenty of fluctuation in the stopping time $I$ itself. As we step backward in time, we expect the fluctuations in $\alpha(G(I_{\delta}))$ should gradually build up, but for small $\delta$ we do not expect these fluctuations to suddenly dominate the fluctuation in $I_\delta$ (when $\delta$ is a constant bounded away from zero, it is easy to see that these two quantities have fluctuations of the same order of magnitude).

The actual proofs of the necessary bounds on the fluctuations of $I_{\delta}$ and $\alpha(G(I_{\delta}))$ can be broken up into two parts. First, there is an analytic part, where we study a certain system of differential equations to estimate the fluctuations of certain statistics of $G(I_{\delta})$. Second, there is a combinatorial part, where we use abstract coupling/concentration techniques to control the \emph{conditional} fluctuations of $\alpha(G(I_{\delta}))$, given certain statistics of $G(I_{\delta})$.
\subsection{Asymptotic analysis of differential equations}\label{subsec:DE-idea}
In their seminal work~\cite{KS81}, Karp and Sipser already understood that for an Erd\H os--R\'enyi random graph $G\sim \mb G(n,c/n)$, at any step $i$ of the leaf-removal process, the distribution of the random graph $G(i)$ can be exactly characterised via certain key statistics of $G(i)$: indeed, if we condition on the number of leaves, the number of vertices of degree at least 2, and the total number of edges of $G(i)$, then $G(i)$ becomes a uniform distribution over graphs with these statistics. So, in order to understand the fluctuations of $\alpha(G(I_{\delta}))$, the first step is to estimate the fluctuations of the degree statistics of $G(I_{\delta})$: we need to show that these fluctuations become negligible (compared to the fluctuations of $I_\delta$) as $\delta\to 0$. The precise statements of these estimates appear as \cref{thm:tau-variance,lem:KS-variance}.

The desired fluctuations can be understood via the central limit theorem described in the last subsection (\cref{thm:KS-CLT}): the Ethier--Kurtz machinery provides certain variance formulas in terms of a system of differential equations associated
with the Karp--Sipser process. So, in order to prove \cref{thm:tau-variance,lem:KS-variance}, we need some asymptotic analysis of this system of differential equations. This analysis is rather delicate, because as the process evolves, the degree statistics actually experience larger and larger fluctuations. The key is that the stopping time $I_\delta$ itself has rather large fluctuations, and these fluctuations explain almost all the fluctuation in the degree statistics. That is to say, we need to show that near the end of the process, the fluctuations in the degree statistics are very strongly correlated with the fluctuations in the number of edges. So, if we stop the process at the point where there are $\delta n$ edges, we have essentially eliminated all fluctuation in all the degree statistics.

To actually prove the necessary correlation estimates on the degree statistics, the key idea is to consider a system of differential equations describing the joint evolution of all degree statistics, and show that, near the end of the process, this can be approximated by a \emph{linear} system of differential equations. We then study the eigenvalues of this linear system, and find that all but one of the eigenvalues are negative. This means that fluctuations are suppressed in all but one direction, so near the end of the process all relevant fluctuations are highly correlated (this can be viewed as an instance of \emph{centre manifold theory}). The details appear in \cref{sec:KS-estimates}.

\subsection{Coupling and concentration}
Given the considerations in the previous section, on the degree statistics of $G(I_\delta)$, it remains to upper-bound the fluctuations of $\alpha(G(I_\delta))$, conditioned on those degree statistics. One of the key insights of this paper is that this can be accomplished using no information about the parameter $\alpha$ except its \emph{smoothness}: whether we have $\alpha=\nu$ or $\alpha=\on{rk}/2$, changing a graph $G$ by a single edge can only ever change $\alpha(G)$ by at most 1.

We take advantage of this smoothness in two ways. First, it is well-known that in a variety of different settings, random variables defined in terms of smooth functions tend to have good concentration around their mean. In particular, using a standard concentration inequality, we prove a lemma (\cref{lem:lipschitz-consequence}) implying that if we condition on any ``reasonable'' outcome of the degree statistics of $G(I_\delta)$, the fluctuations of $\alpha(G(I_\delta))$ (around its conditional expected value) become negligible as $\delta\to \infty$.

The remaining task is to show that the conditional expected value of $\alpha(G(I_\delta))$ itself does not fluctuate very much, using our control over the fluctuations of the degree statistics of $G(I_\delta)$ (discussed in the last subsection). This is accomplished via coupling techniques. Specifically, we prove a lemma (\cref{lem:coupling}) showing that for any two outcomes of the degree statistics of $G(I_\delta)$, which do not differ very much from each other, we can couple the two corresponding conditional distributions of $G(I_\delta)$ in such a way that they differ by few edges. Since we have an upper bound on the fluctuations of the degree statistics of $G(I_\delta)$, we deduce that the conditional distribution (and therefore conditional expected value) of the quantity $\alpha(G(I_\delta))$ does not vary too much over the randomness of these degree statistics.

We remark that both of the above steps are completely non-constructive: we are able to establish bounds on the fluctuations of $\alpha(G(I_\delta))$ and its conditional expected value, without ever needing to precisely understand what the expected value of $\alpha(G(I_\delta))$ actually is (for the latter, we would need to know more about $\alpha$ than its smoothness).

\subsection{Multigraphs and the configuration model}\label{subsec:multigraph-idea} There are a number of details that have been swept under the rug in the above discussion. Perhaps most importantly, while we have mentioned that $G(I_\delta)$ is a uniform distribution given its degree statistics, we have not discussed how to actually get a handle on this uniform distribution (\emph{a priori}, the degree constraints give rise to complicated correlations between the edges).

The solution to this problem is to work with with random \emph{multigraphs} instead of random graphs: instead of an Erd\H os--R\'enyi random graph, we consider a random sequence of pairs of vertices (sampled with replacement). Such random multigraph distributions behave much better under degree conditioning (after conditioning on the degree of every vertex, one obtains a well-known distribution called the \emph{configuration model}, where correlations between edges are easy to understand). Of course, our main theorem is about Erd\H os--R\'enyi random graphs, so we also need some kind of theorem comparing Erd\H os--R\'enyi random graphs with random multigraphs. Theorems of this type were already used in the analyses of the Karp--Sipser process in \cite{AFP98,Kre17}, but due to the non-constructiveness of our proof, we need a somewhat stronger comparison theorem (appearing as \cref{thm:multigraph-comparison}, and proved in \cref{sec:deg-seq}); we deduce this from recent work of Janson~\cite{Jan20}.

\section{Scaffold of the proof}

In this section we break down the proof of \cref{thm:main,thm:rank-CLT} into some key lemmas, many of whose proofs are deferred to later sections. %which will be proved in the rest of the paper.
%This section also serves as an overview for the ideas in the proof and how they come together.
Note that given the results in \cite{Kre17,GKSS}, it suffices to prove \cref{thm:main,thm:rank-CLT}
for $c\le e$.%, but we will leave remarks throughout describing how
%the situation changes in the case $c>e$.

\subsection{A random multigraph model}
For most of the proof, instead of dealing directly with random graphs, we will work with a closely related model of random multigraphs, as follows.
\begin{definition}\label{def:random-multigraph}
Let $\mb G^\ast(n,m)$ be a random
$m$-edge multigraph on the vertex set $\{1,\ldots,n\}$, whose edges are obtained
as a sequence of $m$ uniformly random pairs of (not necessarily distinct)
vertices, sampled with replacement.
\end{definition}

The advantage of $\mb G^\ast(n,m)$ is that if we condition on information about degrees of vertices, the conditional distribution remains tractable (as we will see in \cref{subsec:config-model}). We will need a theorem comparing random graphs and random multigraphs; to state this, we need a notion of graph similarity.
\begin{definition}
\label{def:edit-distance}The \emph{edit distance} $\on d_{\mathrm{E}}(G,G')$
between two multigraphs $G,G'$ (on the same vertex set) is the number
of edges that must be added and removed to obtain one from the other.
\end{definition}

Now, our comparison theorem is as follows.
\begin{theorem}\label{thm:multigraph-comparison}
Fix a constant $c>0$, and consider
the vertex set $\{1,\ldots,n\}$. Consider one of the following two
situations.
\begin{itemize}
\item Let $G$ contain a uniformly random subset of exactly $\floor{cn/2}$
of the possible edges. Let $G^\ast\sim\mb G^\ast(n,\floor{cn/2})$.
\item Let $G$ be a random graph containing each of the $\binom{n}{2}$
possible edges with probability $c/n$ independently. Let $M\sim\on{Bin}(\binom{n}{2},c/n)$
and let $G^\ast\sim\mb G^\ast(n,M)$.
\end{itemize}
Then we can couple $G,G^\ast$ such that $\on d_{\mathrm{E}}(G,G^\ast)$ is bounded
in probability\footnote{Recall that a sequence of random variables $(X_n)_{n=1}^\infty$ is \emph{bounded in probability} or \emph{tight}
if for all $\varepsilon>0$, there are $N,M$ such that $\Pr[X_n\ge M]\le \varepsilon$ for all $n\ge N$.}.
\end{theorem}

That is to say, a random graph with $\floor{cn/2}$ edges can be very
well approximated by $\mb G^\ast(n,\floor{cn/2})$, and a random graph
with edge probability $c/n$ can be very well approximated by first
sampling its number of edges $M\sim\on{Bin}(\binom{n}{2},c/n)$, and then
considering $\mb G^\ast(n,M)$. We prove \cref{thm:multigraph-comparison} in \cref{sec:deg-seq}, using a recent theorem of
Janson~\cite{Jan20}.

For our purposes, the significance of the edit distance is that the rank and matching number are both \emph{Lipschitz functions} with respect to this distance: if we add or
remove a single edge, we change the rank by at most 2, and the matching
number by at most 1. So, writing $\alpha(G)=\nu(G)$ or $\alpha(G)=\on{rk}(G)/2$,
we have
\begin{equation}
|\alpha(G)-\alpha(G')|\le \on d_{\mathrm{E}}(G,G').\label{eq:lipschitz}
\end{equation}

We remark that there are many other similar theorems comparing different models of random graphs and multigraphs. In particular, to prove a central limit theorem for a ``binomial'' random graph (where every edge is present with some probability $p$ independently), a common approach would be to prove a central limit theorem for a random graph with a fixed number of edges, and then ``integrate'' that central limit theorem over possible numbers of edges (this approach appeared for example in the previous work of Pittel~\cite{Pit90} and Krea\v ci\'c~\cite{Kre17}). Due to the non-constructiveness of our proof (briefly mentioned in the introduction), 
this approach is actually not possible for us, and we need the more direct coupling in \cref{thm:multigraph-comparison}.

\subsection{A stopped central limit theorem}

As discussed in \cref{subsec:stop-idea}, our approach is to stop the process at the first time $I_\delta$ when there are at most $\delta n$ edges remaining. For constant $\delta>0$, it is straightforward to adapt the techniques
in \cite{Kre17} to prove a central limit theorem for $I_\delta$. The specific result
we need is as follows (recall the random multigraph $\mb G^\ast(n,m)$
from \cref{def:random-multigraph}).
\begin{lemma}\label{thm:KS-CLT}
Fix constants $0<c\le e$ and $\delta>0$, with $\delta$
sufficiently small in terms of $c$. Let $G$ be a random graph defined
in one of the following two ways:
\begin{itemize}
\item let $M\sim\on{Bin}(\binom{n}{2},c/n)$ and let $G\sim\mb G^\ast(n,M)$, or
\item let $G\sim\mb G^\ast(n,\floor{cn/2})$.
\end{itemize}
Then, consider the Karp--Sipser leaf-removal process on $G$, and
let $I_{\delta}$ be the first time $i$ for which $G(i)$ has
at most $\delta n$ edges\footnote{It would be more natural to define $I_\delta$ to be the first time $i$ for which $G(i)$ has at most $\delta n$ leaves (this time is always finite, even for $c>e$, and is a more natural generalisation of the stopping time $I$ discussed earlier). However, for technical reasons this stopping time is somewhat less convenient to analyse.} (let $I_\delta=\infty$ if this never happens, say). For some $\mu_{\delta},\sigma_{\delta}$, we have
\[
{\displaystyle \frac{I_{\delta}-\mu_{\delta}}{\sigma_{\delta}}}\overset{d}{\to}\mathcal{N}(0,1).
\]
\end{lemma}

We sketch the proof of \cref{thm:KS-CLT} in \cref{sec:CLT} (it is nearly exactly the same as the proof in \cite{Kre17} for the $c>e$ case of \cref{thm:main}\footnote{To say a bit more: the approach in the regime $c>e$ is to prove a bivariate central limit theorem for $I=I_0$ and the number of vertices $v(G(I))$ in the Karp--Sipser core. One can then deduce a central limit theorem for $\alpha(G)$, using \cref{eq:KS-bound} and a separate result due to Aronson, Frieze and Pittel~\cite[Theorem~3]{AFP98} approximating $\alpha(G(I))$ with $v(G(I))/2$.}).

Recall from \cref{sec:ideas} that our objective is to show that
the fluctuations in $\alpha(G(I_{\delta}))$ are small compared
to the fluctuations in $I_{\delta}$. To this end, we need a lower bound on $\sigma_\delta$, and an upper bound on the fluctuations in $\alpha(G(I_\delta))$. %For both these bounds, we will need to carefully study the system of differential equations used to approximate the Karp--Sipser process, and for the latter bound we will also need a coupling argument taking advantage of the ``smoothness'' of the rank and matching number.% (as observed in \cref{eq:lipschitz}).

First, the following lemma records a lower bound on $\sigma_\delta$.
\begin{lemma}\label{thm:tau-variance}
Fix constants $c\le e$ and $\delta>0$ and let $G,I_{\delta},\mu_\delta,\sigma_\delta$ be as in \cref{thm:KS-CLT}. Then $\sigma_\delta=\Theta_c(\sqrt n)$. 
\end{lemma}

Then, we need to prove an upper bound on the fluctuations of $\alpha(G(I_{\delta}))$. To this end, the first step is to obtain upper bounds on the fluctuations of its degree statistics, as follows.

\begin{lemma}\label{lem:KS-variance}
Fix constants $c\le e$ and $\varepsilon>0$.
Then there is $\delta>0$ such that the following holds for sufficiently
large $n$. Let $G,I_{\delta}$ be as in \cref{thm:KS-CLT}.
For each $d\ge 1$, let $X^{(d)}$ be the number of degree-$d$ vertices in $G(I_\delta)$. Then there are $\mu^{(1)},\ldots,\mu^{(n)}$ such that:
\begin{enumerate}
    \item $\sum_d d\mu^{(d)}\le \varepsilon n$, and 
    \item writing $D=\sum_{d=1}^\infty d\,|X^{(d)}-\mu^{(d)}|$, we have $\Pr[D>\varepsilon\sqrt n]\le \varepsilon$.
\end{enumerate}
\end{lemma}

As outlined in \cref{subsec:DE-idea}, the proofs of \cref{thm:tau-variance,lem:KS-variance} (which appear in \cref{sec:KS-estimates}) use the same Gaussian convergence
machinery as the proof of \cref{thm:KS-CLT}, and also involve some
rather nontrivial analysis of a system of differential equations associated
with the Karp--Sipser process.% In particular, \cref{lem:KS-variance} (which states that certain fluctuations become negligible as $\delta\to 0$) is especially delicate, because as the process evolves, the degree statistics actually experience larger and larger fluctuations. The key is that the stopping time $I_\delta$ itself has rather large fluctuations, and these fluctuations explain almost all the fluctuation in the degree statistics. That is to say, to prove \cref{lem:KS-variance}, we need to show that near the end of the process, the fluctuations in all the degree statistics are very strongly correlated with the fluctuations in the number of edges. So, if we stop the process at the point where there are $\delta n$ edges, we have essentially eliminated all fluctuation in all the degree statistics.

%To actually prove the necessary correlation estimates on the degree statistics, the key idea is to consider a system of differential equations describing the joint evolution of all degree statistics, and show that, near the end of the process, this can be approximated by a \emph{linear} system of differential equations. We then study the eigenvalues of this linear system, and find that all but one of the eigenvalues are negative. This means that fluctuations are suppressed in all but one direction, so near the end of the process all relevant fluctuations are highly correlated.

\subsection{Coupling configuration models}\label{subsec:config-model}

Now, crucially, as first observed by Karp and Sipser~\cite{KS81}, the degree statistics $X^{(d)}$
are sufficient to describe the distribution of $G(I_{\delta})$: the conditional distribution of $G(I_{\delta})$ given $(X^{(d)})_{d=1}^\infty$ is precisely described by an associated \emph{configuration model}, as follows.
\begin{definition}
For a degree sequence $\mbf d=(d_1,\ldots,d_n)$, consider a set of $r=d_1+\cdots+d_n$ ``stubs'', grouped into $n$ labelled ``buckets'' of sizes $d_1,\ldots,d_n$. A \emph{configuration} is a perfect matching on the $r$ stubs, consisting of $r/2$ disjoint edges. Given a configuration, contracting each of the buckets to a single vertex gives rise to a multigraph with degree sequence $d_1,\ldots,d_n$. For a set $V$ and a degree sequence $\mbf d\in \mb N^V$, let $\mb G^\ast(\mbf d)$ be the random multigraph distribution obtained by contracting a uniformly random configuration.
\end{definition}

\begin{lemma}\label{lem:symmetry}
Let $G$ be as in \cref{thm:KS-CLT}. For any
$i,d$, let $V^{(d)}(i)$ be the set of degree-$d$ vertices
in $G(i)$.

For any stopping time $I$ (with respect to the filtration described
by the $V^{(d)}(i)$), the distribution of $G(I)$ may be described as follows. Let $V=V^{(1)}(I)\cup \cdots \cup V^{(d)}(I)$, and define $\mbf d\in \mb N^V$ by taking $d_v=d$ when $v\in V^{(d)}$. Then $G(I)\sim \mb G^\ast(\mbf d)$.
\end{lemma}
The proof of \cref{lem:symmetry} is fairly immediate (we can view $G$ as a uniformly random sequence of edges, and a random multigraph $G^*\sim \mb G^*(\mbf d)$, with its edges randomly ordered, can be interpreted as a uniformly random sequence of edges constrained to have degree sequence $\mbf d$). Alternatively, \cref{lem:symmetry} follows directly from \cite[Lemma~2]{AFP98}.

Recall that our goal is to upper-bound the fluctuation in $\alpha(G(I_\delta))$, using the upper bounds on fluctuations of degree statistics in \cref{lem:KS-variance}. To this end, we need the following coupling lemma
for random configurations: if we have two degree sequences $\mbf d,\mbf d'$ which are statistically similar, then we can couple the corresponding configuration models $\mb G^*(\mbf d),\mb G^*(\mbf d')$ to be close with respect to edit distance (recall the definition of edit distance
from \cref{def:edit-distance}).
\begin{lemma}
\label{lem:coupling}Fix two degree sequences $\mbf d=(d_1,\ldots,d_{n})$ and $\mbf d'=(d_1',\ldots,d_{n}')$.
Then we can couple $G\sim \mb G^\ast(\mbf d)$ and $G'\sim \mb G^\ast(\mbf d')$ such that $\on d_{\mathrm{E}}(G,G')\le \sum_v |d_v-d_v'|+1$.
\end{lemma}

\begin{proof}
Let $d^{\mr{max}}_v=\max (d_v,d_v')$ for all $v\in[n]$ (we then increase some $d^{\mr{max}}_v$ by one, if necessary, to ensure that $\sum_v d^{\mr{max}}_v$ is even). Let $\mbf d^{\mr{max}}=(d^{\mr{max}}_v)_v\in \mb N^n$.

Now, consider $\sum_v d^{\mr{max}}_v$ stubs, grouped into buckets of sizes $d^{\mr{max}}_v$. For each vertex $v$, label $d^{\mr{max}}_v-d_v$ of its stubs as being ``$\mbf d$-bad'', and label $d^{\mr{max}}_v-d_v'$ of stubs as being ``$\mbf d'$-bad''. If a stub is not $\mbf d$-bad we say it is ``$\mbf d$-good'' and if a stub is not $\mbf d'$-bad we say it is ``$\mbf d'$-good'' So, a perfect matching of all the stubs is a configuration for $\mbf d^{\mr{max}}$, a perfect matching of all the $\mbf d$-good stubs is a configuration for $\mbf d$, and a perfect matching of all the $\mbf d'$-good stubs is a configuration for $\mbf d'$.

Starting from a uniformly random configuration $\pi^{\mr{max}}$ for $\mbf d^{\mr{max}}$, we can obtain a random configuration $\pi$ for $\mbf d$ as follows. First, delete all matching edges involving a $\mbf d$-bad stub. Some $\mbf d$-good stubs are now unmatched; choose a uniformly random perfect matching of these unmatched $\mbf d$-good stubs to extend our matching to a configuration $\pi$ for $\mbf d$. By symmetry, $\pi$ is a uniformly random configuration for $\mbf d$, and $\on d_{\mr{E}}(\pi^{\mr{max}},\pi)$ is at most the number of $\mbf d$-bad stubs, which is $\sum_v(d^{\mr{max}}_v-d_v)$.

In the same way, starting from $\pi^{\mr{max}}$ we can obtain a uniformly random configuration $\pi'$ for $\mbf d'$ with $\on d_{\mr{E}}(\pi^{\mr{max}},\pi')\le \sum_v(d^{\mr{max}}_v-d_v')$. By the triangle inequality we then have
\[\on d_{\mr{E}}(\pi,\pi')\le \sum_v(d^{\mr{max}}_v-d_v)+\sum_v (d^{\mr{max}}_v-d_v')\le \sum_v|d_v-d_v'|+1,\]
as desired.
\end{proof}

\subsection{Smoothness of the rank and matching number}

Recall from \cref{eq:lipschitz} that the rank and matching number are both
\emph{Lipschitz functions} in terms of edit distance, that is, 
\[
|\alpha(G)-\alpha(G')|\le \on d_{\mathrm{E}}(G,G').
\]
This implies that $\alpha(G)$ is well-concentrated for random $G$. We need this fact for a few different models of random (multi)graphs $G$, as follows.
\begin{lemma}\label{lem:lipschitz-consequence}
Let $\alpha$ be any graph parameter satisfying \cref{eq:lipschitz}.% The following hold:
\begin{enumerate}
\item Let $G$ be a random graph as in \cref{thm:KS}. Then $\alpha(G)$ is subgaussian\footnote{We say that a random variable $X$ is \emph{subgaussian} with \emph{variance
proxy} $\nu$ if $\Pr[|X-\mb E X|\ge x]\le O(\exp(-x^2/\nu))$ for all
$x\in\mb R$.} with variance proxy $O_{c}(n)$.
\item Let $G\sim\mb G^\ast(\mbf d)$ for some degree sequence $\mbf d=(d_1,\ldots,d_n)$. Then
$\alpha(G)$ is subgaussian with variance proxy $O(d_1+\cdots+d_n)$.
\end{enumerate}
\end{lemma}
\begin{proof}
First, (1) is easily proved with 
the Azuma--Hoeffding inequality (see for example the appendix of \cite{BLS11}).

For (2), note that a uniformly random configuration can be defined in terms of a uniformly random permutation $\sigma$ of length $N:=d_1+\cdots+d_n$. Indeed, consider $d_1+\cdots+d_n$ stubs divided into buckets of sizes $d_1,\ldots,d_n$ and consider the configuration (i.e., perfect matching on the stubs) with edges \[\sigma(1)\sigma(2),\,\sigma(3)\sigma(4),\,\ldots,\,\sigma(N-1)\sigma(N).\] Modifying $\sigma$ by a transposition results in a change of at most two edges of our random configuration (which changes $\alpha(G)$ by at most 2), so the desired result follows from a version of the Azuma--Hoeffding inequality for random permutations (see for example \cite[Theorem~5.2.6]{Ver18}).
\end{proof}

\subsection{Completing the proof}

We now show how to combine all the relevant ingredients to prove \cref{thm:main,thm:rank-CLT}.
\begin{proof}[Proof of \cref{thm:main,thm:rank-CLT}]
Let $\alpha(G)$ be $\nu(G)$ or $\on{rk}(G)/2$. Recall that we only need to handle the case $c\le e$. We fix a constant $c\le e$
throughout this proof (implicit constants in asymptotic notation are
allowed to depend on $c$).

First, a minor technical remark: what we will prove is that
there are some $\mu,\sigma$ such that $(\alpha(G)-\mu)/\sigma\overset{d}{\to}\mathcal{N}(0,1)$.
\emph{A priori}, there may be no connection between $\mu$ and $\mb E[\alpha(G)]$
or between $\sigma^{2}$ and $\operatorname{Var}[\alpha(G)]$, if the mean or variance
of $\alpha(G)$ is dominated by the effect of outliers. However, such
pathological behaviour is ruled out by \cref{lem:lipschitz-consequence}(1).

Then, observe that, using \cref{thm:multigraph-comparison,eq:lipschitz}, it suffices to consider $G$ drawn from one of the two random multigraph models in \cref{thm:KS-CLT}
(instead of the two random graph models in \cref{thm:KS}). Indeed, we can assume that the $O(1)$ edit-distance error arising from \cref{thm:multigraph-comparison,eq:lipschitz} is negligible relative to the fluctuation in $\alpha(G)$: note that the conclusion of \cref{thm:KS-CLT} can only hold for $\sqrt{\on{Var}[\alpha(G)]}\to \infty$, because $\alpha(G)$ only takes values in the lattice $(1/2)\mb Z$.

Now, recall that convergence in distribution is metrisable: for example, the \emph{Prohorov} metric $\on d_{\mathrm{P}}$ on real probability distributions
is such that $X_{n}\overset{d}{\to}Z$ if and only if $\on d_{\mathrm{P}}(X_{n},Z)\to0$ (see for example \cite[Section~6]{Bill99}). Recall the definitions of $I_\delta$ and $D=\sum_{d=1}^\infty d\,|X^{(d)}-\mu^{(d)}|$ from \cref{thm:KS-CLT,lem:KS-variance}, let $\mu_\delta,\sigma_\delta$ be as in \cref{thm:KS-CLT}, let $\delta=\delta(\varepsilon)>0$ be
as in \cref{lem:KS-variance}, and note that \cref{thm:KS-CLT} and
\cref{lem:KS-variance} together imply that for any $\varepsilon>0$
we have
\begin{equation}
\on d_{\mathrm{P}}\left(\frac{I_{\delta}-\mu_{\delta}}{\sigma_{\delta}},\;\mathcal{N}(0,1)\right)\le\varepsilon\text{ and }\Pr[D>\varepsilon\sqrt{n}]\le\varepsilon\label{eq:CLT-constant-error}
\end{equation}
for sufficiently large $n$ (say, $n\ge n_{\varepsilon}$). Since
this holds for any constant $\varepsilon>0$, one can abstractly show that in fact \cref{eq:CLT-constant-error}
holds for some $\varepsilon=o(1)$ (decaying with $n$). Indeed, for $n\ge n_{1}$, let
$k_{n}=\max\{k\colon n\ge n_{1/k}\}$, so \cref{eq:CLT-constant-error} holds
for $\varepsilon=1/k_{n}=o(1)$. We have proved that 
\begin{equation}
\frac{I_{\delta}-\mu_{\delta}}{\sigma_{\delta}}\overset{d}{\to}\mathcal{N}(0,1)\;\text{ and }\;\frac{D}{\sqrt{n}}\overset{p}{\to}0\label{eq:d-p}
\end{equation}
(where now $\delta=\delta(\varepsilon)$ depends on $n$ via $\varepsilon$).

The second part of \cref{eq:d-p} says that there is some $\rho=o(1)$ such that whp\footnote{We say an event holds \emph{with high probability}, or \emph{whp} for short, if it holds with probability $1-o(1)$.} $D\le \rho \sqrt n$. That is to say, writing $\mathcal{X}$ for the set of all sequences $(x^{(d)})_{d=1}^\infty\in \mb Z_{\ge 0}^{\mb N}$ such that $\sum_{d=1}^\infty d|x^{(d)}-\mu^{(d)}|\le \rho \sqrt n$, we have $(|X^{(d)}|)_{d=1}^\infty\in \mc X$ whp (recall that $X^{(d)}$ is the number of degree-$d$ vertices that remain at time $I_\delta$).

For each $\mbf{x}=(x^{(d)})_{d=1}^\infty\in\mathcal{X}$,
let $\mathcal{E}_{\mbf{x}}$ be the event that $(|X^{(d)}|)_{d=1}^\infty=\mbf x$
and let $G_{\mbf{x}}\sim\mb G^\ast(\mbf d)$, where $\mbf d$ is a degree sequence containing $x^{(d)}$ copies of each $d$. By \cref{lem:symmetry},
up to relabelling vertices, the conditional distribution of $G(I_{\delta})$
given $\mathcal{E}_{\mbf{x}}$ is precisely that of $G_{\mbf{x}}$
(so, in particular, $\mb E[\alpha(G(I_{\delta}))\,|\,\mathcal{E}_{\mbf{x}}]=\mb E[\alpha(G_{\mbf{x}})]$).
Note that the number of vertices in $G_{\mbf x}$ is $\sum_{d=1}^\infty x^{(d)}\le \varepsilon n+\rho \sqrt n\le 2\varepsilon n$ by \cref{lem:KS-variance}(1).

Recalling that $\varepsilon=o(1)$, by \cref{lem:lipschitz-consequence}(2) we have
\begin{equation}
\Pr\left[\left|\vphantom{I^{I}}\alpha(G(I_{\delta}))-\mb E[\alpha(G_{\mbf{x}})]\right|\ge\varepsilon^{1/3}\sqrt{n}\,\middle|\,\mathcal{E}_{\mbf{x}}\right]=O(\exp(-\Omega(\varepsilon^{-1/3})))=o(1).\label{eq:concentration-within}
\end{equation}
Now, consider $\mbf{x},\mbf{x}'\in\mathcal{X}$, with corresponding degree sequences $\mbf d$ and $\mbf d'$. Reorder $\mbf d,\mbf d'$ such that $d_v=d_v'$ for as many vertices as possible, so $\sum_v |d_v-d_v'|\le \sum_d d\,|x^{(d)}-x^{(d')}|\le 2D=o(\sqrt n)$. Then, \cref{lem:coupling}
tells us that we can couple $G_{\mbf{x}}$ and $G_{\mbf{x}'}$
such that $\on d_{\mathrm{E}}(G_{\mbf{x}},G_{\mbf{x}'})=o(\sqrt{n})$,
meaning that $|\mb E[\alpha(G_{\mbf{x}})]-\mb E[\alpha(G_{\mbf{x}'})]|=o(\sqrt{n})$
(recalling \cref{eq:lipschitz}). Since this is true for each $\mbf x,\mbf x'\in \mc X$, it follows that there is some $\mu_{\mr{res}}$
such that $|\mb E[\alpha(G_{\mbf{x}})]-\mu_{\mr{res}}|=o(\sqrt{n})$ for
each $\mbf{x}\in\mathcal{X}$. (We remark that we do not actually know the value of $\mu_{\mr{res}}$; in this sense our proof is ``non-constructive'').

Combining this with \cref{eq:concentration-within}, we deduce that
\[
\frac{\alpha(G(I_{\delta}))-\mu_{\mr{res}}}{\sqrt{n}}\overset{p}{\to}0.
\]
Finally, recalling from \cref{thm:tau-variance} that $\sigma_{\delta}=\Theta(n^{1/2})$, and recalling the first part of \cref{eq:d-p}, we see that
\[
\frac{\alpha(G)-\mu_{\delta}-\mu_{\mr{res}}}{\sigma_{\delta}}\overset{d}{\to}\mathcal{N}(0,1)
\]
since $\alpha(G)=\alpha(G(I_\delta))+I_\delta$, as desired.
\end{proof}

\section{Reduction to Multigraphs}\label{sec:deg-seq}
In this section we prove \cref{thm:multigraph-comparison}. For $G,G^\ast$ as in \cref{thm:multigraph-comparison}, note that the distribution of $G$ may be obtained from the distribution of $G^\ast$ simply by conditioning on the event that $G^\ast$ is a simple graph (note that every $m$-edge simple graph is an equally likely outcome of $\mb G^\ast(n,m)$). Our proof of \cref{thm:multigraph-comparison} has two parts. First, by an estimate of McKay and Wormald~\cite{MW91}, conditioning on simplicity does not significantly bias the distribution of the degree sequence. Second, by a result of Janson~\cite{Jan20}, given a particular degree sequence $\mbf d$, we can efficiently couple the configuration model $\mb G^\ast(\mbf d)$ with a random graph constrained to have degree sequence $\mbf d$.

We start with the (very simple) fact that in a sparse random graph with average degree about $cn$, the maximum degree is at most $\log n$, and the second factorial-moment of the degrees is about $c^2 n$.
\begin{definition}
For a constant $c$, say that a degree sequence $\mbf d=(d_1,\ldots,d_v)$ is $(n,c)$-\emph{good} if
\begin{itemize}
    \item $\max_{v}d_v\le \log n$,
    \item $\left|\sum_{v} d_v-cn\right|\le n^{3/4}$,
    \item $\left|\sum_{v} d_v (d_v-1)-c^2 n \right|\le n^{3/4}$.
\end{itemize}
\end{definition}
\begin{lemma}\label{lem:almost-all-good}
Fix a constant $c>0$ and let $G,G^\ast$ be as in \cref{thm:multigraph-comparison}. Then whp the degree sequences of $G,G^\ast$ are both $(n,c)$-good.
\end{lemma}
\begin{proof}
It is well-known (see for example \cite[Lemma~1]{BCFF00}) that the degree sequence of $G$ and $G^\ast$ can both be obtained by considering a sequence of $n$ independent $\on{Poisson}(c)$ random variables, and conditioning on an event that holds with probability $\Omega_c(1/\sqrt n)$. Then the desired result follows from a Chernoff bound (noting that if $Q\sim \on{Poisson}(c)$ then $\mb E[Q(Q-1)]=c^2$).
\end{proof}

The following estimate on the simplicity probability follows from, for example, \cite[Lemma~5.1]{MW91}.
\begin{lemma}\label{lem:simple-prob}
    Suppose $\mbf d$ is an $(n,c)$-good degree sequence, and let $G^\ast\sim \mb G^\ast(\mbf d)$. Then 
    \[\Pr[G^\ast\emph{ is simple}]=\exp(-c/2-c^2/4)+o(1).\]
\end{lemma}

Then, we need the following consequence of \cite[Theorems~2.1 and~3.2]{Jan20}, due to Janson.

\begin{lemma}\label{thm:janson-coupling}
Suppose $\mbf d$ is an $(n,c)$-good degree sequence. Let $G^\ast\sim \mb G^\ast(\mbf d)$, and let $G$ be a uniformly random graph on the vertex set $\{1,\ldots,n\}$ with degree sequence $\mbf d$. Then we can couple $G,G^\ast$ such that $\on d_{\mathrm{E}}(G,G^\ast)$ is bounded
in probability.
\end{lemma}

Now we prove \cref{thm:multigraph-comparison}.
\begin{proof}[Proof of \cref{thm:multigraph-comparison}]
    Let $\mbf D,\mbf D^\ast$ be the (random) degree sequences of $G$ and $G^\ast$, and recall that $G$ can be obtained by conditioning on simplicity of $G^\ast$. For any $(n,c)$-good degree sequence $\mbf d$, using \cref{lem:simple-prob} we have
    \[\Pr[\mbf D=\mbf d]=\frac{\Pr[\mbf D^\ast=\mbf d]\Pr[G^\ast\text{ is simple}\,|\,\mbf D^\ast=\mbf d]}{\Pr[G^\ast\text{ is simple}]}=\frac{\exp(-c/2-c^2/4)+o(1)}{\Pr[G^\ast\text{ is simple}]}\Pr[\mbf D^\ast=\mbf d].\]

    By \cref{lem:almost-all-good}, the sum of $\Pr[\mbf D=\mbf d]$ over all good $\mbf d$ and the sum of $\Pr[\mbf D^\ast=\mbf d]$ over all good $\mbf d$ are both $1-o(1)$. So, the above equation implies that
    \[\frac{\exp(-c/2-c^2/4)+o(1)}{\Pr[G^\ast\text{ is simple}]}=1+o(1),\]
    meaning that for each good $\mbf d$ we have $\Pr[\mbf D=\mbf d]=(1+o(1))\Pr[\mbf D^\ast=\mbf d]$.

    So, we can couple $\mbf D,\mbf D^\ast$ to be equal whp. Then, given outcomes of $\mbf D,\mbf D^\ast$, we have $G^\ast\sim \mb G^\ast(\mbf D^\ast)$, and $G$ is a uniformly random graph with degree sequence $\mbf D$, so the desired result follows from \cref{thm:janson-coupling}.
\end{proof}

\section{Analysis of the Karp--Sipser process}\label{sec:CLT}
In this section we sketch how to prove \cref{thm:KS-CLT} using the approach in Krea\v ci\'c's thesis~\cite{Kre17} (a stochastic generalisation of the differential equations method). As we will see, we do not require any change to the proof approach in \cite{Kre17}; we simply need to change the quantity we are interested in estimating (so, the reader may wish to refer to \cite{Kre17} for more details). The concepts and notation in this section will also be important for the proof of \cref{lem:KS-variance}, which will appear in the next section.

Where convenient, we use the same notation as in \cite{Kre17} (in particular, contrary to the preceding sections, for objects that depend on $n$ we explicitly write a superscript $n$, and we do not use boldface for vectors). Throughout this section we fix a constant $c$ (implicit constants in asymptotic notation are
allowed to depend on $c$).

\subsection{A general framework for distributional approximation of Markov chains}\label{subsec:general-CLT}

Before we begin to discuss the details of the proof of \cref{thm:KS-CLT}, we orient the unfamiliar reader with the general framework of Ethier and Kurtz~\cite[Chapter~11]{EK86} for approximating sequences of (time- and space-) inhomogeneous Markov Chains by Gaussian processes. We stress that this should be treated merely as an outline/sketch; more details can be found in \cite[Section~2.3]{Kre17}.

One way to characterise ($d$-dimensional) Brownian motion is as a scaling limit of a sequence of unbiased random walks on the integer lattice $\mb Z^d$. In \cite[Chapter~11]{EK86}, Ethier and Kurtz situated this in a much more general framework. Given a collection of ``rate functions'' $\beta_l:\mb R\to \mb R_{\ge 0}$ (for $l\in \mb Z^d$), for each $n\in \mb N$ we define a space-inhomogeneous random walk $(U^n(s))_{s\ge 0}$ in $ \mb Z^d$: when $U^n(s)$ is located at position $k\in \mb Z^d$, steps in direction $l$ are taken with rate $n\beta_l(k/n)$. Under appropriate assumptions, Ethier and Kurtz proved that, as $n\to \infty$, a sequence of random walks of this type converges to a certain Gaussian process which can be described as the solution to a stochastic partial differential equation involving the functions $\beta_l$.

In more detail: define the ``drift'' function $F(y)=\sum_{l\in \mb Z^d}\beta_l(y)l$, and let\footnote{Departing slightly from \cite{Kre17}, we use the letter $s$ to denote time ``in the continuous world'' and $t$ to denote time ``in the discrete world'' (these essentially differ by a factor of $n$).} $(\chi(s))_{s\ge 0}$ be the solution to the differential equation $\chi(s)=U^n(0)/n+\int_0^s F(\chi(q))\,dq$. Under appropriate assumptions, we expect $U^n(s)$ to be approximately equal to $n \chi(s)$: indeed, the position of our particle at time $s$ is approximately the accumulation of the rate functions until time $s$, taking into account the changes in these rate functions as the particle moves through space. This function $\chi$ is sometimes known as the \emph{fluid limit approximation} of our discrete-time process. The method of obtaining this function and showing that it indeed approximates $U^n(s)$ (in probability) is (an instance of) the \emph{differential equation method} in combinatorics.

Next, for $x\in \mb R^d$, define the matrix $\partial F(x)\in \mb R^{d\times d}$ by $(\partial F(x))_{i,j}=\partial F_{i}(x)/\partial x_{j}$, and for each $l\in\mathbb{Z}^{d}$, let $W_{l}\colon[0,\infty)\to\mb R$ be a standard Brownian motion. Then, define the (continuous-time) random process $(V(s))_{s\ge 0}$ as the solution to the SPDE
\begin{equation}\label{eq:spde} V(s)=V(0)+\sum_{l\in\mb Z^d}W_{l}\left(\int_{0}^{s}\beta_{l}(\chi(q))dq\right)l+\int_0^s \partial F(\chi(q))V(q)dq,
\end{equation}
where $V(0)$ is a Gaussian random variable approximating the initial fluctuations of $U^n(0)/n$ (we allow for the possibility that the particle starts in a random position).
Under appropriate assumptions, we expect $(U^n(s)-n\chi(s))/\sqrt n$ to be approximately distributed as $V(s)$ (jointly for all $s$ up to any fixed time horizon). Indeed, the first term captures the Gaussian fluctuation remaining after approximating the accumulation of random steps by its fluid limit approximation. The second term captures the fluctuation in the accumulation of the drift function $F$ itself, due to the fact that the drift ``should really'' be evaluated at $U^n(q)/n$, and not the deterministic vector $\chi(q)$ (the idea is that $\partial F(\chi(q)) V(q)$ describes the extent to which the fluctuation in $V(q)$ affects this difference). The validity of this approximation is the main content of the Ethier--Kurtz framework.

It is possible to solve the SPDE in \cref{eq:spde} in terms of a Gaussian process and a function $\Phi$ which is itself defined in terms of a deterministic differential equation (i.e., we can describe $V(s)$ ``explicitly'', instead of in terms of the solution to an SPDE). Specifically, for $s,u\ge 0$, let $\Phi(s,u)\in \mb R^{d\times d}$
be the matrix solution to the system of differential equations
\[
\frac{\partial}{\partial s}\Phi(s,u)=\partial F(\chi(s))\Phi(s,u),
\]
with boundary conditions $\Phi(u,u)=I$ for all $u\ge 0$. This matrix function $\Phi$ can be thought of as a ``temporal correlation function'' measuring
the extent to which fluctuations at time $u$ influence fluctuations
at some later time $s$. Under appropriate assumptions, the solution to the SPDE in \cref{eq:spde} is given by
\begin{equation}\label{eq:V-prelim}
V(s)=\Phi(s,0)V(0)+\int_{0}^{s}\Phi(s,u)\,d\mathcal{W}(u),
\end{equation}
where
\[
\mathcal{W}(s)=\sum_{l\in\mathbb{Z}^{4}}W_{l}\left(\int_{0}^{s}\beta_{l}(\chi(q))dq\right)l.
\]
(Note that $V(s)$ is a Gaussian process since it is an It\^o integral with respect to a Gaussian process of a deterministic function.)

Also, we remark that from the description in \cref{eq:V-prelim} and the Ethier--Kurtz approximation theorem, it is possible to deduce the approximate distribution of $U^n$ at certain stopping times: for example, if for some $\delta\ge 0$ and some coordinate $i\le d$, we define $\tau^n_\delta$ to be the (random) first point in time where $U^n_i\le \delta n$, and $s_\delta$ to be the minimum value of $s$ such that $\chi(s)\le \delta$, then (under appropriate assumptions) we have
\begin{equation}
    \frac{U^{n}(\tau_{\delta}^{n})-n\chi(s_{\delta})}{\sqrt{n}}\overset{d}{\to}V(s_{\delta})-\frac{V_{i}(s_{\delta})}{F_{i}(\chi(s_{\delta}))}F(\chi(s_\delta))
\end{equation}
(i.e., we ``subtract away'' the fluctuations in $U^n$ due to the fluctuations in the stopping time $\tau^n_\delta$ itself). \textbf{This last step is the critical place where degeneracy issues can cause problems:} specifically, we will run into problems if $F_{3}(\chi(s_{\delta}))=0$ (in which case, attempting to adjust for the fluctuation in $\tau^n_\delta$ amounts to an indeterminate division $0/0$).

In the rest of this section, we will sketch the application of the framework described above (with certain small technical modifications) to study the Karp--Sipser process and prove \cref{thm:KS-CLT}. As we will see, the only place where we depart from \cite{Kre17} is that we use a different stopping time $\tau_\delta^n$, chosen to avoid the type of degeneracy described above.

\subsection{Setup for the Markov Chain}
To be able to easily apply the Ethier--Kurtz machinery, it is convenient to study the Karp--Sipser process in continuous
time, according to a ``Poisson clock'': at time zero, a leaf-removal
takes place, and then after each leaf-removal we wait for an Exponential(1)
amount of time before the next leaf-removal. For $t\ge0$:
\begin{itemize}
\item let $X_{1}^{n}(t)$ be the number of vertices with degree 1 at time
$t$,
\item let $X_{2}^{n}(t)$ be the number of vertices of degree at least 2
at time $t$,
\item let $X_{3}^{n}(t)$ be the number of edges at time $t$, and
\item let $X_{4}^{n}(t)$ be the number of leaf-removal steps up until time
$t$.
\end{itemize}
Then, let $X^{n}(t)=(X_{1}^{n}(t),X_{2}^{n}(t),X_{3}^{n}(t),X_{4}^{n}(t))$.
As observed in \cite[Lemma~2]{AFP98} (restated as \cite[Lemma~5]{Kre17}; see also \cite[Section~3]{KS81}), $X^{n}(t)$ is a
continuous-time Markov chain. 
Let $\tau_{\delta}^{n}=\inf\{t\ge0:X_{3}^{n}(t)\le\delta n\}$ (with $\tau_\delta^n=\infty$ if $X_{3}^{n}(t)> \delta n$ for all $t$); we
will study the evolution of $X^{n}$ until time $\tau_{\delta}^{n}$.

Note that the random variable $I_\delta$ in \cref{thm:KS-CLT} is precisely $X_4(\tau_\delta^n)$ with this setup (with the convention that $X_4^n(\infty)=\infty$).

\subsection{Rate and drift functions}\label{subsec:rate-drift}

To understand the evolution of $X^n(t)$, we need to study
the transition probabilities corresponding to a single
random leaf-removal.

When we delete a leaf $v$ together with its neighbour
$w$, we also delete all other edges incident to $w$, and all the
neighbours of $w$ get their degree reduced by 1. The effect this
has on $X_{1}^{n}(t),X_{2}^{n}(t),X_{3}^{n}(t)$ can be described
in terms of the degrees of $w$ and its neighbours (with respect to
the multigraph $G^{n}(t)$ remaining at time $t$). In order to study
the distribution of these quantities, one can explicitly describe
the conditional distribution of $G^{n}(t)$ given $X_{1}^{n}(t),X_{2}^{n}(t),X_{3}^{n}(t)$:
this conditional distribution of $G^{n}(t)$ is the same as the conditional distribution of $\mb G^\ast(X_{1}(t)+X_{2}(t),\,X_{3}(t))$, given that there are $X_{1}(t)$
vertices of degree 1 and $X_{2}(t)$ vertices of degree at least 2. One can study the typical degree
distribution of this random graph, and then use standard techniques
for studying random graphs with given degree sequences.

In a bit more detail: it turns out that the degree distribution in such a graph $G^n(t)$ can be described by a truncated Poisson distribution with a certain parameter depending on $X_1^n, X_2^n$, and $X_3^n$. Indeed, for $x=(x_{1},x_{2},x_{3},x_{4})\in(\mb{R}_{\ge 0})^4$ with $x_2>0$ and $2x_3\ge x_1+2x_2$, let $z(x)\ge 0$ be
the unique solution\footnote{In \cite{Kre17}, Krea\v ci\'c writes ``$z_x$''; we have changed the notation for readability.
} to 
\[
\frac{z(x)(e^{z(x)}-1)}{e^{z(x)}-z(x)-1}=\frac{2x_{3}-x_{1}}{x_{2}},
\]
and let $Z^{n}(t)=z(X^{n}(t))$. Then, consider any time $t$ where the process is ``not too degenerate'' (in notation to be introduced in \cref{subsec:fluid-limit}, we can take any $t\le\tau_{\delta}^{n}\land\zeta_{\delta}^{n}$), and condition on a corresponding outcome of $X^{n}(t)$. For any of the $X_{2}^{n}(t)$ vertices with degree at least 2, and any
$d\le\log n$, the probability the degree of that vertex is exactly
$d$ is 
\[
\Pr[Q=d\,|\,Q\ge2]+o\left(\frac{\log^{3}n}{n}\right),
\]
where $Q\sim\operatorname{Poisson}(Z^{n}(t))$. That is to say, the
degree distribution is roughly ``truncated Poisson'' with parameter $Z^{n}(t)$. (One can show that whp no vertex ever
has degree greater than $\log n$, so it suffices to consider $d\le\log n$).
The above fact appears as \cite[Lemma~22]{Kre17}\footnote{This is stated for $t\le\tau^{n}\land\zeta^{n}$ (for $c>e$) instead
of $t\le\tau_{\delta}^{n}\land\zeta_{\delta}^{n}$, but the proof
is exactly the same: the only role $\zeta^{n}$ plays is to ensure
that $X_{1}^{n}(t),X_{2}^{n}(t),X_{3}^{n}(t)$ are of size $\Omega(n)$.} (deduced from \cite[Lemma~5]{AFP98}). Using this fact, the transition probabilities for a single leaf-removal
step are computed in \cite[Theorem~28]{Kre17}:
given that $(X_{1}^{n}(t),X_{2}^{n}(t),X_{3}^{n}(t))=(q_{1},q_{2},q_{3})$
for some $t\le\tau_{\delta}^{n}\land\zeta_{\delta}^n$, we consider
the conditional probability that the next-removed leaf $v$ has a
neighbour $w$ which itself has $k_{1}\ge 1$ degree-1 neighbours, $k_{2}$ degree-2
neighbours and $k_{3}$ neighbours of degree at least 3 (which causes
$(X_{1}^{n},X_{2}^{n},X_{3}^{n})=(q_{1}-k_{1}+k_{2},\;q_{2}-1-k_{2},\;q_{3}-k_{1}-k_{2}-k_{3})$ if $k_1+k_2+k_3\ge 2$, and $(X_{1}^{n},X_{2}^{n},X_{3}^{n})=(q_{1}-2,\;q_{2},\;q_{3}-1)$ if $k_1+k_2+k_3=1$).
This conditional probability is shown to be a somewhat complicated formula involving $q_{1},q_{2},q_{3},k_{1},k_{2},k_{3}$,
plus an additive error term of the form $o(\log^{3}n/n)$.

This estimate for the 1-step transition probabilities translates into an estimate for the transition rates
of the continuous-time Markov chain $X^n(t)$: specifically, the approximate transition
rate from a state $q\in\mb Z_{\ge 0}^{4}$ to a state $q+l\in\mb Z_{\ge 0}^{4}$ is $n\beta_l(q/n)$, for rate functions $\beta_l$ defined as follows. Let
\[\mc K=\{(k_1,k_2,k_3):k_1\ge 1,\;k_2,k_3\ge 0,\;k_1+k_2+k_3\ge 2\}\]
and for $l=(-k_1+k_2,-1-k_2,-k_1-k_2-k_3,1)$ with $(k_1,k_2,k_3)\in \mc K$, let $\beta_l(x)$ be
\[\frac1{(k_1-1)!k_2!k_3!}\cdot\frac{x_2}{2x_3}\cdot\frac{z(x)}{e^{z(x)}-z(x)-1}\left(\frac{x_1}{2x_3}z(x)\right)^{k_1-1}\left(\frac{x_2z(x)^2}{2x_3(e^{z(x)}-z(x)-1)}z(x)\right)^{k_2}\left(\frac{x_2z(x)}{2x_3}z(x)\right)^{k_3}.\]
Also, let $\beta_{(-2,0,-1,1)}(x)=x_1/(2x_3)$, and let $\beta_l(x)=0$ for all other $l$. 

Taking a weighted
sum of the approximate transition rates $\beta_{l}(x)$ allows us to estimate
the expected infinitesimal change to $X^{n}(t)$: as in \cite[Equation~2.46]{Kre17} we can define the ``drift function'' 
\[
F(x)=\sum_{l\in\mb Z^{4}}\beta_{l}(x)l=(F_{1}(x),F_{2}(x),F_{3}(x),F_{4}(x))
\]
and compute
\begin{align}\label{eq:drift-1}
F_1(x) &= -1-\frac{x_1}{2x_3}+\frac{x_2^2z(x)^4e^{z(x)}}{(2x_3f(z(x)))^2}-\frac{x_1x_2z(x)^2e^{z(x)}}{(2x_3)^2f(z(x))}\\
F_2(x) &= -1+\frac{x_1}{2x_3}-\frac{x_2^2z(x)^4e^{z(x)}}{(2x_3f(z(x)))^2}\label{eq:drift-2}\\
F_3(x) &= -1-\frac{x_2z(x)^2e^{z(x)}}{2x_3f(z(x))}\label{eq:drift-3}\\
F_4(x) &= 1.\notag
\end{align}
where $f(x)=e^x-x-1$.

\subsection{Fluid limit approximation}\label{subsec:fluid-limit}
Given the rate functions $\beta_l$, we can now solve a system of differential equations to obtain\footnote{As described in \cref{subsec:general-CLT}, it is possible to prove a fluid limit approximation by studying a differential equation involving the drift functions $F$. However, this is not actually done in  Krea\v ci\'c's thesis~\cite{Kre17}: it is more convenient simply to cite the previous work in \cite{AFP98} (which has slightly different notation and a different formulation of the relevant differential equations).} a \emph{fluid limit approximation} $\chi$ for $X^n$. In this section we record formulas for this fluid limit approximation (and record the theorem that the trajectory of $X^n$ does indeed concentrate around this approximation).

Let $p(u)=e^{-u}(e^u-u-1)$, let $\beta(u)$ be the unique solution to $\beta(u)e^{c\beta(u)}=e^u$ and implicitly define the function $\vartheta\colon[0,\infty)\to \mb R$ by
\[s=\frac1c\left(c(1-\beta(\vartheta(s)))-\frac12 \log^2\beta(\vartheta(s))\right).\] Then let
\begin{align*}
    \chi_1(s)&=\frac1c\left(\vartheta^2(s)-\vartheta(s)\cdot c\cdot \beta(\vartheta(s))(1-e^{-\vartheta(s)})\right),\\
    \chi_2(s)&=p(\vartheta(s))\beta(\vartheta(s)),\\
    \chi_3(s)&=\frac1{2c}\vartheta^2(s)\\
    \chi_4(s)&=s.
\end{align*}
The following result is presented as \cite[Theorem~20]{Kre17}, as a consequence of estimates in \cite{AFP98}. It holds for all $c>0$.

\begin{lemma}\label{thm:AFP-fluid}
With notation as defined in this section, for
$s^\ast=\inf\{s\ge0\colon\chi_{1}(s)=0\}$ and any $s<s^\ast$ (not depending
on $n$) we have
\[
\sup_{u\le s}\left\Vert \frac{X^{n}(nu)}{n}-\chi(u)\right\Vert _{\infty}\overset{p}{\to}0.
\]
\end{lemma}

For $c\le e$, let $s_{\delta}=\inf\{s\ge0\colon\chi_{3}(s)\le\delta\}<s^\ast$ be the
``fluid limit prediction'' for $\tau_{\delta}^n$. One can check (e.g., using the series expansions we will compute in \cref{subsec:fluid-reparameterisation}) that $s_\delta$ is finite and well-defined, and that $\chi_{1}(s_{\delta}),\chi_{2}(s_{\delta})>0$ (i.e., at time $\tau_{\delta}^n$,
there are likely to be $\Omega_\delta(n)$ vertices of degree at least 2,
and therefore $\Omega_\delta(n)$ edges). Let 
\[
\zeta_{\delta}^{n}=\inf\{t\ge0:X_{1}^{n}(t)\le\chi_{1}(s_{\delta})n/2\text{ or }X_{2}^{n}(t)\le\chi_{2}(s_{\delta})n/2\},
\]
so whp $\tau_{\delta}^{n}\le\zeta_{\delta}^{n}$ (that is to say,
whp $X_{1}^{n}(t),X_{2}^{n}(t),X_{3}^{n}(t)$ are whp of size $\Omega_\delta(n)$
until time $\tau_{\delta}^n$).
\begin{remark}
The stopping times $\tau_\delta^n$ and $\zeta_\delta^n$ are the main point of difference between our setting and Krea\v ci\'c's thesis~\cite{Kre17}. In \cite{Kre17}, the hitting times $\tau^{n}=\inf\{t\ge0:X_{1}^{n}(t)=0\}$
and $\zeta^{n}=\inf\{t\ge0:X_{2}^{n}(t)\le\chi_{2}(s^\ast)/2\text{ or }X_{3}^{n}(t)\le\chi_{3}(s^\ast)/2\}$ are instead considered; when $c>e$, one can check that $\chi_{2}(s^\ast),\chi_{3}(s^\ast)>0$, and $\tau^n\le\zeta^n$ whp.
\end{remark}

\subsection{Initial fluctuations}

Knowing the transition rates for our Markov chain is essentially tantamount to knowing the full distribution of $(X^n(t))_{t\ge0}$,
except that we also need to estimate the (asymptotically multivariate
Gaussian) distribution of the initial state $X^n(0)$. This routine
computation is performed in \cite[Section~2.3.2]{Kre17} in
the case where $G\sim\mb G^\ast(n,\floor{cn/2})$; we also need a similar calculation in the case where
$G\sim\mb G^\ast(n,M)$ with $M\sim\on{Bin}(\binom{n}{2},c/n)$. This is the only place where the distinction
between our two random graph models actually has an impact.
\begin{lemma}\label{lem:init-fluctuations}
Fix $c>0$.
\begin{itemize}
\item If $G\sim\mb G^\ast(n,\floor{cn/2})$, then
\[
\left(\frac{X_{1}^{n}(0)-n\chi_{1}(0)}{\sqrt{n}},\;\frac{X_{2}^{n}(0)-n\chi_{2}(0)}{\sqrt{n}},\;\frac{X_{3}^{n}(0)-n\chi_{3}(0)}{\sqrt{n}},\;\frac{X_{4}^{n}(0)-n\chi_{4}(0)}{\sqrt{n}}\right)
\]
converges in distribution to a multivariate Gaussian distribution
with mean zero and covariance matrix 
\[
\Sigma=\begin{pmatrix}c^{2}e^{-2c}+ce^{-c}-ce^{-2c}-c^{3}e^{-2c} & -ce^{-c}+ce^{-2c}+c^{3}e^{-2c}&0&0\\
-ce^{-c}+ce^{-2c}+c^{3}e^{-2c} & (e^{-c}+ce^{-c})(1-e^{-c}-ce^{-c})-c^{3}e^{-2c}&0&0\\
0&0&0&0\\
0&0&0&0\\
\end{pmatrix}.
\]
\item If $G\sim\mb G^\ast(n,M)$ with $M\sim\on{Bin}(\binom{n}{2},c/n)$ then
\[
\left(\frac{X_{1}^{n}(0)-n\chi_{1}(0)}{\sqrt{n}},\;\frac{X_{2}^{n}(0)-n\chi_{2}(0)}{\sqrt{n}},\;\frac{X_{3}^{n}(0)-n\chi_{3}(0)}{\sqrt{n}},\;\frac{X_{4}^{n}(0)-n\chi_{4}(0)}{\sqrt{n}}\right)
\]
converges in distribution to a multivariate Gaussian distribution
with mean zero and covariance matrix 
\[
\Sigma=\begin{pmatrix}(c^3-3c^2+c)e^{-2c}+ce^{-c} & (-c^3+2c^2+c)e^{-2c}-ce^{-c} & (c-c^2)e^{-c} & 0\\
(-c^3+2c^2+c)e^{-2c}-ce^{-c} & (c^3-c^2-2c-1)e^{-2c}+(c+1)e^{-c} & c^{2}e^{-c} & 0\\
(c-c^2)e^{-c} & c^{2}e^{-c} & c/2 & 0\\
0 & 0 & 0 & 0
\end{pmatrix}.
\]
\end{itemize}

\end{lemma}
We remark in order to prove the central limit theorems in \cref{thm:main,thm:rank-CLT} we need to know that initial fluctuations are asymptotically Gaussian, but if we are not interested in knowing the asymptotic variance of the rank or matching number it is not necessary to know the actual value of $\Sigma$.

\begin{proof}[Proof of \cref{lem:init-fluctuations}]
Since the $\mb G^\ast(n,\floor{cn/2})$ case was already considered in \cite[Theorem~39]{Kre17}, we just consider the case where $G\sim\mb G^\ast(n,M)$ with $M\sim\on{Bin}(\binom{n}{2},c/n)$. By \cref{thm:multigraph-comparison}, it actually suffices to consider the number $X_1$ of degree-1 vertices, the number $X_2$ of vertices of degree at least 2, and the number $X_3$ of edges, in an $n$-vertex Erd\H os--R\'enyi random graph with edge probability $c/n$.

Let $X^{(d)}$ be the number of vertices with degree $d$ in such a random graph. So,  writing $a_1^{(d)}=\mbm 1_{d=1}$, $a_2^{(d)}=\mbm 1_{d\ge 2}$ and $a_3^{(d)}=d/2$, we have $X_j=\sum_{d=0}^\infty a_j^{(d)}X^{(d)}$ for $j\in\{1,2,3\}$.
In \cite[Theorem~4.1]{Jan07}, Janson finds the asymptotic joint distribution of $(X^{(d)})_{d=0}^\infty$: namely,
\[\left(\frac{X^{(d)}-\mb E X^{(d)}}{\sqrt n}\right)_{d=0}^\infty\overset d \to (U^{(d)})_{d=0}^\infty,\]
where the $U^{(d)}$ are jointly Gaussian with
\[\on{Cov}[U^{(d)},U^{(d')}]=\pi(d)\pi(d')\left(\frac{(d-c)(d'-c)}c-1\right)+\mbm 1_{d=d'} \pi(d),\]
where $\pi(d)=e^{-c}c^d/d!$. As discussed in \cite[Theorem~4.1]{Jan07}, this convergence in distribution behaves well with respect to (potentially infinite) linear combinations of the $U^{(d)}$, as long as the coefficients do not grow super-exponentially fast: specifically, letting $U_j=\sum_{d=0}^\infty a_j^{(d)}U^{(d)}$, we have
\[\left(\frac{X_j-\mb E X_j}{\sqrt n}\right)_{j\in\{1,2,3\}}\overset d \to (U_1,U_2,U_3),\]
with $\on{Cov}[U_j,U_{j'}]=\sum_{d=0}^\infty\sum_{d'=0}^\infty a_j^{(d)}a_{j'}^{(d')}\on{Cov}[U^{(d)},U^{(d')}]$. The desired result then follows from a routine (if tedious) calculation\footnote{Accompanying the arXiv version of this paper, we include a Mathematica script that performs this calculation.}.
\end{proof}

\subsection{Approximation by a Gaussian process}\label{subsec:gaussian-approximation}

Now, using all the estimates we have obtained so far, we can apply the general framework described in \cref{subsec:general-CLT},
to show that the normalised process
\[
\left(\frac{X^{n}(ns)-n\chi(s)}{\sqrt{n}}\right)_{s\ge0}
\]
is well-approximated by a certain Gaussian process $V$ (defined in terms
of the initial covariance matrix $\Sigma$ in \cref{lem:init-fluctuations} and the approximate transition
rates $\beta_{l}(x)$, via the drift function $F$). 

Consequently, we can show that 
\[
\frac{X^{n}(\tau_{\delta}^{n})-n\chi(s_{\delta})}{\sqrt{n}}
\]
converges to a certain multivariate Gaussian distribution, which proves \cref{thm:KS-CLT}.

In more detail: as in \cref{subsec:general-CLT}, define the matrix
$\partial F(x)\in \mb R^{4\times 4}$ by $(\partial F(x))_{i,j}=\partial F_{i}(x)/\partial x_{j}$. Independently for each $l\in\mathbb{Z}^{4}$
let $W_{l}\colon[0,\infty)\to\mb R$ be a standard Brownian motion, and define
the Gaussian processes
\[
\mathcal{W}(s)=\sum_{l\in\mathbb{Z}^{4}}W_{l}\left(\int_{0}^{s}\beta_{l}(\chi(q))dq\right)l.
\]
Let
\[\Phi\colon\{(s,u):0\le u\le s< s^\ast\}\to\mb R^{4\times4}\]
be the unique matrix solution to the system of differential equations
\[
\frac{\partial}{\partial s}\Phi(s,u)=\partial F(\chi(s))\Phi(s,u),
\]
with boundary conditions $\Phi(u,u)=I$ for all $0\le u< s^\ast$
(here $I$ is the $4\times4$ identity matrix).

Let $V(0)$ be a 4-variate Gaussian random variable with covariance
matrix $\Sigma$ as in \cref{lem:init-fluctuations}, and define the Gaussian process
\begin{equation}\label{eq:V}
V(s)=\Phi(s,0)V(0)+\int_{0}^{s}\Phi(s,u)\,d\mathcal{W}(u).
\end{equation}
As (informally) described in \cref{subsec:general-CLT}, we can use the Gaussian process $V$ to describe the fluctuations in the combinatorial process $X^{n}$. First, for
any $s<s^\ast$ we have the convergence in distribution of processes
\[
\left(\frac{X^{n}(nu)-n\chi(u)}{\sqrt{n}}\right)_{0\le u\le s}\overset{d}{\to}(V(u))_{0\le u\le s}.
\]
Second, (assuming $c\le e$) we have $F_3(\chi(s_\delta))<0$, and we deduce a similar convergence in distribution at the stopping time $\tau_{\delta}^n<\infty$:
\begin{equation}
    \frac{X^{n}(\tau_{\delta}^{n})-n\chi(s_{\delta})}{\sqrt{n}}\overset{d}{\to}V(s_{\delta})-\frac{V_{3}(s_{\delta})}{F_{3}(\chi(s_{\delta}))}F(\chi(s_\delta)),\label{eq:gaussian-convergence}
\end{equation}
which yields \cref{thm:KS-CLT}. To orient the reader relative to the writeup in \cite{Kre17}: Equation~\cref{eq:gaussian-convergence} is essentially proved in \cite[Theorem~38]{Kre17} (with $s^*$ in place of $s_\delta$, in the regime $c>e$)\footnote{Notice that $s_\delta,\tau_\delta^n$ are defined in terms of $\chi_3,X_3^n$, whereas in \cite{Kre17}, $s^\ast,\tau^n$ are defined in terms of $\chi_1,X_1^n$. This explains the role of $V_3,F_3$ in \cref{eq:gaussian-convergence} versus $V_1,F_1$ in \cite{Kre17}.}.
\section{Estimating fluctuations}\label{sec:KS-estimates}
In this section we prove \cref{thm:tau-variance,lem:KS-variance} via the machinery discussed in \cref{sec:CLT}. Throughout this section, we again fix a constant $c\le e$ (implicit constants in asymptotic notation are again allowed to depend on $c$).

For \cref{thm:tau-variance} we need to study the number of steps $I_\delta$ until there are at most $\delta n$ edges remaining. In the notation of \cref{sec:CLT}, this number of steps is precisely $X_4^n(\tau_\delta^n)$. 

For \cref{lem:KS-variance}, we need to study certain degree statistics. These can be studied in terms of the quantities $X_1^n(\tau_\delta^n),X_2^n(\tau_\delta^n),X_3^n(\tau_\delta^n)$ from \cref{sec:CLT} (namely, the number of leaves, the number of vertices of degree at least 2, and the number of edges, at the first point where there are at most $\delta n$ edges remaining). All of these statistics are small when $\delta$ is small (they measure quantities of the remaining graph very close to the end of the process), and we expect them to have small fluctuations.

So, for both \cref{thm:tau-variance,lem:KS-variance}, the main challenge is to estimate various parameters of the limiting distribution of $X^n(\tau_\delta^n)$ described in \cref{eq:gaussian-convergence}. Specifically, we will prove the following lemma.

\begin{lemma}\label{lem:analytic-estimates}
Let $\chi$ be the fluid limit approximation defined in \cref{subsec:fluid-limit}, let $s_\delta$ be the fluid limit prediction for $\tau_\delta^n$, and let $\Sigma_{\delta}$ be the covariance matrix of the limiting random vector in \cref{eq:gaussian-convergence}. Let $z(x)$ be defined as in \cref{subsec:rate-drift}.
\begin{enumerate}
    \item $\chi_1(s_\delta)=O(\delta)$ and $\chi_j(s_\delta)=\Theta(\delta)$ for $j\in\{2,3\}$.
    \item
    $z(\chi(s_\delta)))=\Theta(\sqrt \delta)$.
    \item $\lim_{\delta\to0}(\Sigma_\delta)_{j,j}= 0$ for $j\in\{1,2,3\}$.
    \item $\liminf_{\delta\to0}(\Sigma_\delta)_{4,4}>0$.
\end{enumerate}
\end{lemma}
\cref{lem:analytic-estimates}(4) directly implies \cref{thm:tau-variance}, because the quantity $\sigma_\delta^2$ in \cref{thm:tau-variance} is asymptotic to $(\Sigma_\delta)_{4,4}n$. We will deduce \cref{lem:KS-variance} from \cref{lem:analytic-estimates}(1--3) at the end of this section (in \cref{subsec:degree-estimates}) using some estimates of Aronson, Frieze, and Pittel~\cite{AFP98} (we need similar considerations as we informally discussed at the start of \cref{subsec:rate-drift}). Meanwhile, most of this section will be spent proving \cref{lem:analytic-estimates}.

\cref{lem:analytic-estimates}(1--2) follow from quite routine computations using the explicit formulas for the fluid limit approximation. The real challenge is to prove (3) and (4) (bounding certain variances). At a very high level, the goal is to prove that near the end of the process, there is very strong correlation between the fluctuations of $X_1^n,X_2^n,X_3^n$ (so if we stop at a time $\tau_\delta^n$ which fixes the value of $X_3^n$, then we essentially eliminate the fluctuation in $X_1^n,X_2^n$), and to prove that there is only weak correlation between the fluctuations of these three statistics and of $X_4^n$.

To be a bit more specific, the machinery outlined in \cref{sec:CLT} gives us an approximation of our process $X^n$ in terms of a Gaussian process $V$. Our goal is to prove that (in the $\delta\to 0$ limit) the covariance matrix $\Sigma(s_\delta)$ of $V(s_\delta)$ has rank 2, and its first three rows and columns comprise a rank-1 submatrix. That is to say, $V_1(s_\delta),V_2(s_\delta),V_3(s_\delta)$ are essentially multiples of each other, while $V_4(s_\delta)$ has plenty of additional variance. The correction in \cref{eq:gaussian-convergence} for the stopping time $\tau_\delta^n$ then yields the desired conclusions.

Recalling the definitions in \cref{sec:CLT}, the covariance matrix of $V(s)$ can be written in terms of (an integral involving) the matrix-valued correlation function $\Phi$ (see \cref{eq:V}), so our task is to obtain a good understanding of $\Phi$. For example, for (3), we want to show that near the end of the process the first three rows and columns of $\Phi$ form a matrix that is nearly rank-1. While $\Phi$ is defined in terms of a system of nonlinear partial differential equations (see \cref{eq:Phi}), near the end of the process this system is quite well-approximated in terms of an (inhomogeneously time-dilated) system of \emph{linear} differential equations (basically, we just need to know the limiting value of $\partial F$ as we approach the end of the process). We can then obtain our desired conclusions by studying the limiting eigensystem of $\partial F$. In particular, for (3), negative eigenvalues in this eigensystem give us exponential decay in all directions but one.

\subsection{Estimating the fluid limit}\label{subsec:fluid-reparameterisation} 
First we prove \cref{lem:analytic-estimates}(1--2) and some estimates that will be used in the proofs of \cref{lem:analytic-estimates}(3--4). For all of these, it is convenient to re-parameterise the fluid limit approximation $\chi$ in terms of $z=z(\chi(s))$. Indeed, as observed in \cite[Lemma~8]{AFP98}, the fluid limit approximation can be re-expressed as
\begin{align*}
\chi_1(s) &= \frac1c(z^2-zc\beta(z)(1-e^{-z})),\\
\chi_2(s) &= (1-(1+z)e^{-z})\beta(z),\\
\chi_3(s) &= \frac1{2c}z^2,\\
\chi_4(s) &=s= \frac1c\left(c(1-\beta(z))-\frac12\log(\beta(z))^2\right).
\end{align*}
(where, as in \cref{subsec:fluid-limit}, $\beta$ is defined implicitly by $\beta(u)e^{c\beta(u)}=e^u$). Also, we have $z\to 0$ as $s\to s^\ast$ (this is only true because we are assuming $c\le e$; if $c>e$ then $z\to z^\ast$ for some explicit $z^\ast>0$ computed in \cite{AFP98}).

Recall the Lambert $W$ function, satisfying $W(t)e^{W(t)}=t$, and note that we can rewrite $\beta(z)$ as $W(ce^z)/c$. Note also that $W(t) > 0$ for $t > 0$, and $W(e) = 1$.

A direct computation with the series expansion for $W$ (see for example \cite{CJK97}) shows that

\begin{align}
\chi_1(s) &= \frac{1-W(c)}c z^2+O(z^3),\notag\\ 
\chi_2(s) &= \frac{W(c)}{2c}z^2+O(z^3),\notag\\
\chi_3(s) &= \frac1{2c}z^2,\notag\\
\chi_4(s) &=s= \frac{2c-2W(c)-W(c)^2}{2c}-\frac{1}{2c(W(c)+1)}z^2+O(z^3).\label{eq:time-z}
\end{align}
\cref{lem:analytic-estimates}(1) and (2) immediately follow: from the definition of $s_{\delta}$ we have that $\chi_3(s_{\delta}) = \delta = \frac{1}{2c}z^2$ yielding (2); plugging in this value of $z$ yields (1). This computation also yields 
\begin{equation}s^\ast=\frac{2c-2W(c)-W(c)^2}{2c}.\label{eq:s*}\end{equation}

We next collect several estimates on the drift function and its partial derivatives near $z = 0$ that follow from the above equations. The necessary computations, while routine, are a bit tedious; accompanying the arXiv version of this paper, we include a Mathematica script that performs these calculations.

For the rest of this section, we generalise the notation ``$O(f)$'' to denote any matrix or vector whose entries are all of the form $O(f)$ (so we can write matrix equations with error terms).

We begin by estimating the drift function near its limit at $s^*$. The following estimate can be obtained by plugging the series expansion of $\chi(s)$ in \cref{eq:time-z} into the formulas for $F_1,F_2,F_3$ in Equations \cref{eq:drift-1,eq:drift-2,eq:drift-3}:
\begin{fact}\label{fact:F}
\[\begin{pmatrix}F_1(\chi(s))\\F_2(\chi(s))\\F_3(\chi(s))\end{pmatrix}=-(1+W(c))\begin{pmatrix}2-2W(c)\\
W(c)\\
1
\end{pmatrix}+O(z).\]
\end{fact}

Let $\hat{F} = (F_1, F_2, F_3)$ be the first 3 coordinates of the drift function, and let
\begin{equation}
    v_0 = (2- 2W(c), W(c), 1)^\intercal\label{eq:v0}
\end{equation}
be the direction of the first 3 coordinates of the drift function near time $\chi(s^*)$. We want to show that at time $s_{\delta}$, $(X_1^n, X_2^n, X_3^n)$ is approximately parallel to $v_0$, by showing that any fluctuations in $(X_1^n, X_2^n, X_3^n)$ in directions orthogonal to $v_0$ are suppressed. To do this, we will need to understand the partial derivatives of $\hat{F}$, which govern how the drift function changes due to fluctuations about $\hat{F}(\chi(s))$. Let $\partial \hat{F}(x)$ be the $3\times 3$ matrix defined by $(\partial \hat F(x))_{i,j}=\partial \hat F_{i}(x)/\partial x_{j}$. The following lemma shows that for $s$ near $s^*$, $\partial\hat{F}(\chi(s))$ is approximately a matrix with one zero eigenvalue, corresponding to the eigenvector $v_0$, and two negative eigenvalues. This ensures that near $s^*$, if (in addition to some fluctuation $\gamma v_0$ parallel to $v_0$) there is a large fluctuation $v$ orthogonal to $v_0$, there will be a tendency for that orthogonal fluctuation to be suppressed, because $v^{\intercal} \hat{F}(X(s)) \approx v^{\intercal} \hat{F}(\chi(s) + \gamma v_0 + v)  \approx v^{\intercal} \partial \hat{F}(\chi(s)) v < 0$.

\begin{lemma}\label{lemma:partial}
\[\partial\hat{F}(\chi(s))=\frac{1}{z^{2}}Q D Q^{-1}+O\left(\frac{1}{z}\right),\]
where
\[
D=c(1+W(c))\begin{pmatrix}0 & 0 & 0\\
0 & -3 & 0\\
0 & 0 & -2
\end{pmatrix},\qquad Q=\begin{pmatrix}2-2W(c) & \frac{1}{2}(W(c)-1)(3W(c)-1) & -4W(c)\\
W(c) & \frac{1}{2}(-2W(c)-1)(W(c)-1) & 2W(c)+1\\
1 & 1 & 1
\end{pmatrix}.
\]
\end{lemma}
\cref{lemma:partial} may be obtained by explicitly computing partial derivatives of the formulas for $F$ in \cref{subsec:rate-drift}, using implicit differentiation to compute $\partial z(x)/\partial x_i$, then using the formulas in \cref{eq:time-z} for $\chi(s)$ in terms of $z$, and then computing a series expansion.

In \cref{subsec:correlation}, we will use \cref{fact:F} and \cref{lemma:partial} (plus some elementary estimates) to study the correlation function $\Phi(s,u)$ (for $s$ near $s^*$). We will also use the following two facts, which are easily verified from the series expansions in \cref{eq:time-z}. 

\begin{fact}\label{fact:chi}
There exists a constant $c_1 > 0$ such that for any $c_2 > 0$ small enough and any $s^\ast - c_1 \leq s \leq s^\ast - c_2$, we have
    $z(\chi(s)), \chi_1(s), \chi_2(s) = \Omega(1)$.
\end{fact}
(Here the constants in asymptotic notation are allowed to depend on $c_1,c_2$.)
\begin{proof}
     If $s$ is bounded away from $s^\ast$, then from the formulas in \cref{subsec:fluid-limit}, $\vartheta(s)$ is bounded away from $0$, and thus so are $\chi_1(s), \chi_2(s)$ and $\chi_3(s)$. Thus from \cref{eq:time-z}, $z(\chi(s))$ is also bounded away from $0$.
\end{proof}
(Although it will not be necessary for us, we remark that one can directly check using the formulas in \cref{subsec:fluid-limit} or \cref{subsec:fluid-reparameterisation} that $z(\chi(s)), \chi_1(s), \chi_2(s) = \Omega(1)$ whenever $s-s^*=\Omega(1)$, i.e., the above statement holds for \emph{any} constants $0<c_2\le c_1\le s^*$.)

\begin{fact}\label{fact:sdelta}
$s^*-s=\Theta(z^2)=\Theta(\chi_3(s))$. In particular, $s^*-s_\delta=\Theta(z(\chi(s_\delta))^2)=\Theta(\delta)$.
\end{fact}

\subsection{Formulas for the limiting variance}\label{subsec:limiting-variance}
In this subsection, we show how to reduce the task of proving \cref{lem:analytic-estimates}(3--4), concerning the covariance matrix of $X^n(\tau_{\delta}^n)$, to certain analytic estimates on the functions $\Phi$ and $\beta_{l}$ (which we will prove in the next section).

Recall the definitions of $V$ and $\mathcal{W}$ from \cref{subsec:gaussian-approximation} (in terms of functions $\Phi,\beta_l,F$ and an initial covariance matrix $\Sigma$ that depends on which of the two models of random graphs in \cref{thm:main} we are considering). Let $\Sigma(s)$ be the covariance matrix of $V(s)$, and compute from \cref{eq:V}
\begin{equation}\label{eq:sigma_s}
\Sigma(s)=\Phi(s,0)\,\Sigma\,\Phi(s,0)^{\transpose}+\int_{0}^{s}\sum_{l\in\mathbb{Z}^{4}}\beta_{l}(\chi(u))\,(\Phi(s,u)\,l\,l^{\transpose}\,\Phi(s,u)^{\transpose})\,du.
\end{equation}

\begin{remark}
For the reader not familiar with stochastic calculus: it may be helpful to view a Brownian motion as the limit of a discrete-time random walk, so our It\^o integral can be approximated as a sum of independent increments (these increments have different sizes given by $\Phi$, and the rate they occur at is given by the $\beta_l$). The total variance is then the sum of variances of these increments; taking a limit gives the above integral.
\end{remark}

Recalling \cref{eq:gaussian-convergence} (with the correction for the stopping time $\tau_\delta^n$), we have

\begin{equation}
\Sigma_{\delta}=P_{\delta}\,\Phi(s_\delta,0)\,\Sigma\,\Phi(s_\delta,0)^{\transpose}\,P_{\delta}^{\transpose}+\int_{0}^{s_\delta}\sum_{l\in\mathbb{Z}^{4}}\beta_{l}(\chi(u))\,(P_{\delta}\,\Phi(s_\delta,u)\,l\,l^{\transpose}\,\Phi(s_\delta,u)^{\transpose}P_{\delta}^{\transpose})\,du,
\label{eq:Sigmadelta}
\end{equation}
where
\[
P_{\delta}=I-\frac1{F_{3}(\chi(s_{\delta}))}\begin{pmatrix}0 & 0 & F_{1}(\chi(s_{\delta})) & 0\\
0 & 0 & F_{2}(\chi(s_{\delta})) & 0\\
0 & 0 & F_{3}(\chi(s_{\delta})) & 0\\
0 & 0 & F_{4}(\chi(s_{\delta})) & 0
\end{pmatrix}=\begin{pmatrix}1 & 0 & -F_{1}(\chi(s_{\delta}))/F_{3}(\chi(s_{\delta})) & 0\\
0 & 1 & -F_{2}(\chi(s_{\delta}))/F_{3}(\chi(s_{\delta})) & 0\\
0 & 0 & 0 & 0\\
0 & 0 & -F_{4}(\chi(s_{\delta}))/F_{3}(\chi(s_{\delta})) & 1
\end{pmatrix}.
\]

We can now reduce \cref{lem:analytic-estimates}(3--4) to some more explicit estimates on the functions $\chi,\beta_l,\Phi$. Let \[
A=\begin{pmatrix}1 & 0 & 0 & 0\\
0 & 1 & 0 & 0\\
0 & 0 & 1 & 0
\end{pmatrix},\qquad B=\begin{pmatrix}0 & 0 & 0 & 1
\end{pmatrix},
\]
so $A\Sigma_\delta A^\transpose$ contains the first three rows and columns of $\Sigma_\delta$ and $B\Sigma_\delta B^\transpose$ contains the $(4,4)$-entry of $\Sigma_\delta$. This means that \cref{lem:analytic-estimates}(3--4) is tantamount to the claims that $\lim_{\delta\to0}\|A\Sigma_\delta A^\transpose\|=0$ and $\liminf_{\delta\to0}\|B\Sigma_\delta B^\transpose\|>0$, for any matrix norm $\|\cdot\|$ (all matrix norms are equivalent).

Recalling that the Frobenius norm\footnote{The \emph{Frobenius norm} of a matrix $M$ is the square root of the sum of squares of entries of $M$.} $\|\cdot\|_{\mr F}$ is subadditive and submultiplicative, from \cref{eq:Sigmadelta} we estimate
\begin{equation}\|A\Sigma_\delta A^\transpose\|_{\mr F}\le \|\Sigma\|_{\mr F}\|AP_{\delta}\,\Phi(s_\delta,0)\|_{\mr F}^2+\int_{0}^{s_\delta}\|AP_{\delta}\,\Phi(s_\delta,u)\|_{\mr F}^2\sum_{l\in\mathbb{Z}^{4}}\beta_{l}(\chi(u))\,\|l\|_{2}^{2}\,du.\label{eq:analytic-123}\end{equation}
Roughly speaking, to prove \cref{lem:analytic-estimates}(3) it now suffices to prove that $\sum_{l\in\mathbb{Z}^{4}}\beta_{l}(\chi(u))\|l\|_2^2=O(1)$ for all $u\in [0,s_\delta]$, and that $\|AP_\delta\Phi(s_\delta,u)\|_{\mr F}$ is tiny (e.g., decays with $\delta$) unless $u$ is very close to $s_\delta$.

For \cref{lem:analytic-estimates}(4), consider any vectors $\mathcal L=\{l_1,l_2,l_3,l_4\}$ in $\mb Z^4$. Let $L\in \mb R^{n\times n}$ be the matrix with columns $l_1,l_2,l_3,l_4$, and let $\sigma$ be the least singular value of $L$. Then, the single entry of the matrix $\sum_{l\in \mc L}(B P_{\delta}\,\Phi(s_\delta,u)\,l\,l^{\transpose}\,\Phi(s_\delta,u)^{\transpose}P_{\delta}^{\transpose} B^{\transpose})$ is $\|BP_\delta \Phi(s_\delta,u)L\|_{\mr F}^2$, which is at least 
$\sigma^2\|BP_\delta \Phi(s_\delta,u)\|_{\mr F}^2$.
So, we have
\begin{equation}\|(B\Sigma_\delta B^\transpose)\|_{\mr F}\ge \sigma^2 \int_0^{s_\delta} \|BP_\delta \Phi(s_\delta,u)\|_{\mr F}^2 \inf_{l\in \mc L}\beta_l(\chi(u))\,du.\label{eq:analytic-4}\end{equation}
Roughly speaking, to prove \cref{lem:analytic-estimates}(4) it suffices to prove that, for some specific basis $\mc L$ of the vector space $\mb R^4$ (not depending on $\delta$) and for all $u$ in some interval of length $\Omega(1)$, both $\inf_{l\in \mc L}\beta_l(\chi(u))$ and $\|BP_\delta \Phi(s_\delta,u)\|_{\mr F}^2$ are of the form $\Omega(1)$.

\begin{remark}\label{rem:limiting-variance}It is worth remarking that the quantities in this section can be used to give asymptotic formulas for the variance of $\alpha(G)$ (in the settings of \cref{thm:main,thm:rank-CLT}). Indeed, our proof of \cref{thm:main,thm:rank-CLT} (in the case $c\le e$) shows that $\alpha(G)$ is asymptotically Gaussian with variance asymptotic to $(\lim_{\delta\to 0} (\Sigma_\delta)_{4,4})n$. The calculations in this section (in particular, the computation of $\lim_{\delta\to 0} P_\delta$ which we will see in \cref{lem:PPhi}) can be used to show that $\lim_{\delta\to 0} (\Sigma_\delta)_{4,4}$ is precisely the single entry of the matrix
\[
BP\,\Phi(s^\ast,0)\,\Sigma\,\Phi(s^\ast,0)^{\transpose}\,P^{\transpose}B^{\transpose}+\int_{0}^{s^\ast}\sum_{l\in\mathbb{Z}^{4}}\beta_{l}(\chi(u))\,(B P\,\Phi(s^\ast,u)\,l\,l^{\transpose}\,\Phi(s^\ast,u)^{\transpose}P^{\transpose}B^\transpose)\,du,
\]
where
\[P=\begin{pmatrix}1 & 0 & 2W(c)-2 & 0\\
0 & 1 & -W(c) & 0\\
0 & 0 & 0 & 0\\
0 & 0 & 1+W(c) & 1
\end{pmatrix}.\]
It is not clear whether the asymptotic variance of $\alpha(G)$ can be expressed more explicitly. (The situation for $c>e$, treated in \cite{Kre17}, is essentially the same.)
\end{remark}

\subsection{Estimating the correlation function}\label{subsec:correlation} In this subsection we study $\Phi$ via its defining system of differential equations. Roughly speaking, our goal is to show that $\Phi(s_\delta,u)$ is approximately a rank-2 matrix, whose first 3 rows and columns approximately comprise a rank-1 submatrix spanned by $v_0$ (as defined in \cref{eq:v0}). From this, we will obtain estimates on $AP_\delta\Phi(s_\delta,u)$ and $BP_\delta\Phi(s_\delta,u)$ as foreshadowed in the previous subsection.

Recall from \cref{subsec:gaussian-approximation} that $\Phi$ is the solution to the system of differential equations
\begin{equation}\label{eq:Phi}
    \frac{\partial}{\partial s}\Phi(s,u)=\partial F(\chi(s))\Phi(s,u),
\end{equation}

with boundary conditions $\Phi(u,u)=I$ for each $u$, where $\partial F(x)$ is the matrix defined by $(\partial F(x))_{i,j}=\partial F_{i}(x)/\partial x_{j}$ (and $F$ is the drift function from \cref{subsec:rate-drift}). Note that $\partial F_4/\partial x_j\equiv0$ and $\partial F_j/\partial x_4\equiv0$ for all $j\in \{1,2,3,4\}$, so we can write
\[\Phi(s,u) = \begin{pmatrix}
\hat{\Phi}(s,u) & 0\\
0 & 1
\end{pmatrix},\qquad \frac{\partial }{\partial s} \hat\Phi(s,u) = \partial \hat F(\chi(s)) \hat \Phi(s,u),\qquad \hat \Phi(u,u)=\hat I\]
where $\hat{\Phi}(s,u)=A\Phi(s,u) A^\transpose$ and $\partial \hat{F}(x)=A\partial F(x) A^\transpose$ contain just the first three rows and columns of $\Phi(s,u)$ and $\partial F(x)$, respectively, and $\hat I\in \mb R^{3\times 3}$ is the $3\times 3$ identity matrix. Note that there is never any interaction between different $u$: we can think of $u$ as a parameter indexing a family of differential equations, and $s$ as the independent variable in each differential equation (ranging from $u$ to $s^*$). One could write $\hat \Phi_u(s)$ instead of $\hat \Phi(s,u)$ to emphasise this.

Also note that there is no interaction between the three columns of $\hat \Phi(s,u)$: each of the columns $g_u(s)$ of $\hat \Phi(s,u)$ is separately a solution to the system of differential equations $\partial g(s)/\partial s = \partial \hat F(\chi(s)) g(s)$ (though, the different columns have different initial conditions). In this section we show that in the limit $s\to s^\ast$, the direction of $g(s)$ does not actually depend on the initial conditions (i.e., the columns of $\hat \Phi(s,u)$ are roughly proportional to each other). We formalize this in \cref{lem:Phi-low-rank} below, which states that for $s$ near $s^*$, $g_u(s)$ is nearly proportional to $(2 - 2W(c), W(c), 1)^\transpose$ for all $u$.

Roughly speaking, the reason this holds is that $\partial \hat F(\chi(s^\ast))$ has two negative eigenvalues and one zero eigenvalue (with corresponding eigenvector $v_0:=(2 - 2W(c), W(c), 1)^\transpose$). So, as our differential equation evolves near time $s^\ast$, the mass of $g_u$ diminishes in all directions except the direction of $v_0$. (The eventual mass in direction $v_0$ depends on the initial conditions, defined in terms of $u$ and the particular column of $\hat \Phi(s,u)$ that we're interested in.)

\begin{lemma}\label{lem:Phi-low-rank}
    For some $u\in [0,s^\ast]$, let $g\colon[u,s^\ast]\to \mb R^3$ be a solution to the system of differential equations
    \[\frac d{ds}g(s)=\partial \hat F(\chi(s)) g(s)\]
    satisfying some initial conditions $g(u)$ with $\|g(u)\|_2\le 1$. Then, for all $s\in [u,s^*]$ we have
    \[g(s)=C\begin{pmatrix}
2-2W(c)\\W(c)\\1
\end{pmatrix} + O\left(\left(\frac{s^\ast-s}{s^\ast-u}\right)^{1/4}\right)
\]
for some $C\in \mb R$ (depending on $s,u$ and the initial conditions $g(u)$).
\end{lemma}
(Recall that we have generalised the notation ``$O(f)$'' so that it may describe a \emph{vector} whose entries are of the form $O(f)$.)
\begin{proof}
As in \cref{subsec:fluid-reparameterisation}, we work largely in terms of the reparameterisation $z=z(\chi(s))$. First, we recall from \cref{lemma:partial} the limiting behaviour of $\partial\hat F$ around $z=0$:
\[\partial\hat{F}(\chi(s))=\frac{1}{z^{2}}QDQ^{-1}+O\left(\frac{1}{z}\right),\]
where
\[
D=c(1+W(c))\begin{pmatrix}0 & 0 & 0\\
0 & -3 & 0\\
0 & 0 & -2
\end{pmatrix},\qquad Q=\begin{pmatrix}2-2W(c) & \frac{1}{2}(W(c)-1)(3W(c)-1) & -4W(c)\\
W(c) & \frac{1}{2}(-2W(c)-1)(W(c)-1) & 2W(c)+1\\
1 & 1 & 1
\end{pmatrix}.
\]
Denoting operator norm by $\|\cdot \|_{\rm{op}}$, observe that $\|Q\|_{\rm{op}},\|Q^{-1}\|_{\rm{op}}=O(1)$ (this amounts to the fact that $\det Q\ne 0$). So, $h=Q^{-1} g$ is a solution to a system of differential equations
\[\frac d{ds}h(s)=\left(\frac1{z^2} D+O(1/z)\right)h(s)\]
for some initial conditions $h(0)$ satisfying $\|h(u)\|_2=O(1)$.

Next, we reparameterise the time variable $s$ to obtain a system of differential equations that is (approximately) linear. Namely, we first let $r=s^\ast-s$, so $ds/dr=-1$ and $1/r=2c(W(c)+1)/z^2+O(1/z)$ by \cref{eq:time-z,eq:s*}. This means that
\[\frac d{dr}h(s)=\left(\frac{-1}{2cr(W(c)+1)} D+O(1/\sqrt r)\right)h(s)\](recall \cref{fact:sdelta}). Then, we let $q=-\log r$, so $dr/dq=-r$ and we can write
\[\frac d{dq}h(s)=\left(\frac{1}{2c(W(c)+1)} D+E(s)\right)h(s),\]
for some matrix-valued function $E$ whose entries are at most $O(\sqrt r)=O(e^{-q/2})$. The time interval from $s=u$ to $s=s^\ast$ corresponds to the interval from $r = s^* - u$ to $r = 0$, which corresponds to the interval from $q=q_u:=-\log(s^\ast-u)$ to $q=\infty$.

Now, we are finally ready to study the evolution of our system of differential equations. First, we prove that $h(s)$ does not blow up as $s\to s^\ast$.
\begin{claim}\label{claim:no-blowup}
    $\sup_{s\in [u,s^\ast]}\|h(s)\|_2=O(1).$
\end{claim}
\begin{proof}
Let $\Lambda=(1/(2c(W(c)+1))) D$ (which is a diagonal matrix with diagonal entries $(\lambda_1,\lambda_2,\lambda_3)=(0,-3/2,-1)$). We have
\begin{align*}
\frac{d}{dq} \|h(s)\|_2^2 = \sum_{j=1}^{3} 2h_j(s)\frac d{dq}h_j(s)&=\sum_{j=1}^{3} 2h_j(s)\bigl(\Lambda h(s)+E(s) h(s)\bigr)_j\\
&\le \sum_{j = 1}^3 \left(2\lambda_jh_j(s)^2+2h(s)_j\|E(s)\|_{\mr{op}}\|h(s)\|_2\right)\\
&\le 2\|E(s)\|_{\mr{op}}\|h(s)\|_1\|h(s)\|_2=O(e^{-q/2}\|h(s)\|_2^2).
\end{align*}
(For the final line, we recall that each $\lambda_j$ is nonpositive.) We can rewrite this inequality as
\[\frac d{dq}\log \|h(s)\|_2^2=O(e^{-q/2}).\]
Then, for any $s$, integration yields 
\[\log \|h(s)\|_2^2\le \log \|h(u)\|_2^2+\int_{q_u}^{q}O(e^{-p/2})\,dp= O(1).\qedhere\]
\end{proof}
Next, using \cref{claim:no-blowup}, we prove that $h_2(s)$ and $h_3(s)$ decay as $s\to s^\ast$.
\begin{claim}\label{claim:decay}
For $j\in\{2, 3\}$ we have $h_j(s)^2=O(e^{(q_u-q)/2}).$
\end{claim}
\begin{proof}
Using \cref{claim:no-blowup} and similar calculations as above, for $j\in\{2,3\}$ we have 
\begin{align*}
\frac{d}{dq} h_j(s)^2 = 2h_j(s)\frac d{dq}h_j(s)&\le 2\lambda_jh_j(s)^2+2\|E(s)\|_{\mr{op}}\|h(s)\|_1\|h(s)\|_2\\
&\le -2h_j(s)^2+O(e^{-q/2}\|h(s)\|_2^2)\le -2h_j(s)^2+O(e^{-q/2}).
\end{align*}
(For the last line we used that $\lambda_2,\lambda_3\le -1$.) We can rewrite this inequality as
\[\frac d{dq} (e^{2q} h_j(s)^2)\le e^{2q}\cdot O(e^{-q/2})=O(e^{3q/2}).\]
Integration then yields
\[e^{2q} h_j(s)^2-e^{2q_u}h_j(u)^2\le \int_{q_u}^qO(e^{3p/2})\,dp=O(e^{3q/2}),\]
which implies
\[h_j(s)^2\le O(e^{2(q_u-q)}+e^{-q/2})=O(e^{(q_u-q)/2}).\qedhere\]
\end{proof}
Together \cref{claim:no-blowup,claim:decay} imply that
\[h(s)=Q^{-1}g(s)=C_0\begin{pmatrix}
1\\0\\0
\end{pmatrix} + O(e^{(q_u-q)/4})=C_0\begin{pmatrix}
1\\0\\0
\end{pmatrix}+O\left(\left(\frac{s^\ast-s}{s^\ast-u}\right)^{1/4}\right)\]
for some $C_0\in \mb R$ with $|C_0|=O(1)$.
Multiplying by $Q$ yields the desired result.
\end{proof}

We can now deduce some estimates on $P_\delta \Phi(s_\delta,u)$, suitable for application with \cref{eq:analytic-123,eq:analytic-4}.

\begin{lemma}\label{lem:PPhi}
Recall the definitions of $A,B,P_\delta$ from \cref{subsec:limiting-variance}, and write $f_\delta(u)=((s^\ast-s_\delta)/(s^\ast-u))^{1/4}$.
\begin{enumerate}
    \item $\|AP_\delta \Phi(s_\delta,u)\|_{\mr F}=O(f_\delta(u))$,
    \item $\|BP_\delta \Phi(s_\delta,u)\|_{\mr F}=\Omega(1) - O(f_\delta(u))$.
\end{enumerate}
    
\end{lemma}
\begin{proof}
By \cref{lem:Phi-low-rank}, for each $s,u$ there are real numbers $C_1,C_2,C_3=O(1)$ such that
\[\Phi(s_\delta,u)=\begin{pmatrix}C_{1}(2-2W(c)) & C_{2}(2-2W(c)) & C_{3}(2-2W(c)) & 0\\
C_{1}W(c) & C_{2}W(c) & C_{3}W(c) & 0\\
C_{1} & C_{2} & C_{3} & 0\\
0 & 0 & 0 & 1
\end{pmatrix}+O(f_\delta(u)).\]
Recall from \cref{fact:F} that we have 
\[F(\chi(s))=-(1+W(c))\begin{pmatrix}2-2W(c)\\
W(c)\\
1\\
-1-W(c)
\end{pmatrix}+O(z).\]

By \cref{fact:sdelta} we have $z(\chi(s_\delta))=O(f_\delta(u))$ (for any $u$), so we deduce (using the definition of $P_\delta$) that
\[P_{\delta}=\begin{pmatrix}1 & 0 & -2 + 2W(c) & 0\\
0 & 1 & -W(c) & 0\\
0 & 0 & 0 & 0\\
0 & 0 & 1+W(c) & 1
\end{pmatrix}+O(f_\delta(u)).\]
We then compute 
\begin{align*}AP_{\delta}\Phi(s_{\delta},u)&=O(f_\delta(u)),\\ BP_{\delta}\Phi(s_{\delta},u)&=\begin{pmatrix}C_{1}(1+W(c)) & C_{2}(1+W(c)) & C_{3}(1+W(c)) & 1\end{pmatrix}\;+O(f_\delta(u)).\end{align*}
The desired results follow.
\end{proof}

\subsection{Integrating the correlation function}
Now we combine \cref{lem:PPhi} with some estimates on the transition rates $\beta_l(\chi(s))$, to complete the proofs of \cref{lem:analytic-estimates}(3--4) via the strategy outlined in \cref{subsec:limiting-variance}.
\begin{lemma}\label{lem:beta-lower-bounded}
Define the basis
\[\mc L=\left\{ \begin{pmatrix}-2\\
0\\
-1\\
1
\end{pmatrix},\begin{pmatrix}0\\
-2\\
-3\\
1
\end{pmatrix},\begin{pmatrix}0\\
-3\\
-6\\
1
\end{pmatrix},\begin{pmatrix}1\\
-3\\
-5\\
1
\end{pmatrix}\right\} \]
of $\mb R^4$. For some $c_1 > 0$ and any small enough $c_2 > 0$, for all $s^\ast - c_1 \leq s \leq s^\ast - c_2$, we have $\inf_{l\in \mc L}\beta_l(\chi(s))=\Omega(1)$.
\end{lemma}
\begin{proof}
Recall the notation defined in \cref{subsec:rate-drift}. The corresponding choices of $(k_1,k_2,k_3)\in \mc K$ for the four vectors in $\mc L$ are $(1, 0, 0), (1,1,1),(2,2,2),(1,2,2)$, respectively. Using the formula for $\beta_l(x)$ in \cref{subsec:rate-drift}, for constants $k_1, k_2, k_3$, we have
\[\beta_{l}(x) = \Omega\left(\left(\frac{x_2}{x_3}\right)^{k_2 + k_3 + 1}\right)\cdot \Omega\left(\left(\frac{z}{e^{z}-z-1}\right)^{k_2 + 1}\right)\left(\frac{x_1}{2x_3}\right)^{k_1-1}\!\!\cdot \Omega\left(z^{k_1 + 2k_2 + 2k_3 - 1}\right).\]
\cref{fact:chi} then shows that for  some $c_1 > 0$ and any small enough $c_2 > 0$, we have $z(\chi(s)), \chi_1(s), \chi_2(s) = \Omega(1)$, so $\beta_{l}(\chi(s)) = \Omega(1)$.
\end{proof}
\begin{lemma}\label{lem:beta-sum-bounded}
For all $0\le s\le s^\ast$ we have
\[\sum_{l\in\mathbb{Z}^{4}}\|l\|_{2}^{2}\beta_{l}(\chi(s))=O(1).\]
\end{lemma}
\begin{proof}
Using the approximations in \cref{subsec:fluid-reparameterisation} and the formulas for $\beta_l(x)$ in \cref{subsec:rate-drift}, we can see that $\beta_{(-2,0,-1,1)}(\chi(s))=O(1)$, and for $l=(-k_1+k_2,-1-k_2,-k_1-k_2-k_3,1)$ with $(k_1,k_2,k_3)\in \mc K$,
\[\beta_l(x)=\frac1{(k_1-1)!k_2!k_3!}O\left(\frac1z\right)\left(O(z)\right)^{k_1-1}\left(O(z)\right)^{k_2}\left(O(z^2)\right)^{k_3}=\frac1{(k_1-1)!k_2!k_3!}\left(O(z)\right)^{k_1+k_2+k_3-2}.\]
Indeed, we observe from \cref{eq:time-z} that $x_3 = \Omega(z^2)$, and $x_1, x_2 = O(z^2)$, and then use L'H\^opital's rule on the functions of $z$. Also, for such $l$ we have $\|l\|_2^2=O(k_1^2+k_2^2+k_3^2)$. So,
\[\sum_{l\in\mathbb{Z}^{4}}\|l\|_{2}^{2}\beta_{l}(\chi(s))=O(1)+\sum_{(k_1,k_2,k_3)\in \mc K}O\left(\frac{k_1^2+k_2^2+k_3^2}{(k_1-1)!k_2!k_3!}\right)\left(O(z)\right)^{k_1+k_2+k_3-2}=O(1),\]
recalling that $k_1+k_2+k_3\ge 2$ for all $(k_1,k_2,k_3)\in \mc K$.
\end{proof}
We are now ready to complete the proofs of \cref{lem:analytic-estimates}(3--4).

\begin{proof}
    [Proof of \cref{lem:analytic-estimates}(3)]
Note that $s^\ast-s_\delta=O(\delta)$ (by \cref{fact:sdelta}). For $j\in \{1,2,3\}$, substituting \cref{lem:PPhi}(1) and \cref{lem:beta-sum-bounded} into \cref{eq:analytic-123} yields 
\begin{align*}(\Sigma_\delta)_{j,j}\le \|A\Sigma_\delta A^\transpose\|_{\mr F}&\le O\left(((s^\ast-s_\delta)^{1/4})^2\right)+\int_{0}^{s_\delta}O\left(\left(\left(\frac{s^\ast-s_\delta}{s^\ast-u}\right)^{1/4}\right)^{2}\right)\,du\\
&=O(\delta^{1/2})+O(\delta^{1/2})\int_0^{s^\ast-O(\delta)}\frac{du}{(s^\ast-u)^{1/2}}\\
&=O(\delta^{1/2})\cdot (1+(s^\ast)^{1/2}-O(\delta)^{1/2})=O(\delta^{1/2}),\end{align*}
which tends to zero as $\delta\to 0$.
\end{proof}

\begin{proof}
    [Proof of \cref{lem:analytic-estimates}(4)]
    Let $c_1$ and $c_2$ be the constants in \cref{lem:beta-lower-bounded}, so when $u$ is in the range between $s^\ast - c_1$ and $s^\ast - c_2$ we have $\inf_{l\in \mc L}\beta_l(\chi(s))=\Omega(1)$.
    
    Also, when $u$ is in this range between $s^\ast - c_1$ and $s^\ast - c_2$, we have $f_\delta(u)=O(\delta^{1/4})$ (using \cref{fact:sdelta}, and using that $c_2=\Omega(1)$). So, for sufficiently small $\delta$, for $u$ in this range we have $\|BP_\delta \Phi(s_\delta,u)\|_{\mr F}=\Omega(1)$ by \cref{lem:PPhi}(2).

    We can then directly substitute the above two estimates into \cref{eq:analytic-4}. We obtain $(\Sigma_\delta)_{4,4}=\|B\Sigma_\delta B^\transpose\|_{\mr F}=\Omega(1)$ as desired.
\end{proof}

\subsection{Fluctuations of the degree sequence}\label{subsec:degree-estimates}
Now, having proved \cref{lem:analytic-estimates}, it remains to deduce \cref{lem:KS-variance} from \cref{lem:analytic-estimates}(1--3).

Recall that $X_1^n(\tau_\delta^n),X_2^n(\tau_\delta^n),X_3^n(\tau_\delta^n)$ measure the number of leaves, the number of vertices of degree at least 2, and the number of edges, at the first time $\tau_\delta^n$ where there are at most $\delta n$ edges remaining. \cref{lem:analytic-estimates}(3) (together with the Gaussian approximation \cref{eq:gaussian-convergence}) allows us to control the fluctuations of these statistics.

In the statement of \cref{lem:KS-variance}, $X^{(d)}$ is the number of degree-$d$ vertices at time $\tau_\delta^n$ (so in particular $X^{(1)}=X_1^n(\tau_\delta^n)$ and $\sum_{d=2}^\infty X^{(d)}=X_2^n(\tau_\delta^n)$). To prove \cref{lem:KS-variance}, we study the fluctuations of these degree statistics $X^{(d)}$ \emph{conditional on} $X_1^n(\tau_\delta^n),X_2^n(\tau_\delta^n),X_3^n(\tau_\delta^n)$.

To this end, we use estimates of Aronson, Frieze, and Pittel~\cite{AFP98}, already mentioned informally at the start of \cref{subsec:rate-drift}. Specifically, the following lemma is a slight strengthening of \cite[Lemma~5]{AFP98}\footnote{In \cite[Lemma~5]{AFP98}, the notation ``$v$'' is used instead of ``$k_2$'', and the notation ``$X_j$'' is used for the degree of a vertex $j$.}, and follows from the same proof (the estimate for $\Pr[\deg(v)=\deg(v')=d]$ below is just slightly less wasteful than in the statement of \cite[Lemma~5]{AFP98}).
\begin{lemma}\label{lem:degree-prob}
For some $k_1,k_2,k_3\in \mb N$, let $G$ be a random multigraph $\mb G^\ast(k_1+k_2,k_3)$ conditioned on the event that the vertices in $V_1:=\{1,\ldots,k_1\}$ have degree exactly 1, and the vertices in $V_2:=\{k_1+1,\ldots,k_1+k_2\}$ have degree at least 2. Let $f(z)=e^z-z-1$ and let $z$ be the unique solution to
\[
\frac{z(e^{z}-1)}{f(z)}=\frac{2k_{3}-k_{1}}{k_{2}},
\]
(so $z$ is precisely $z(k_1,k_2,k_3,0)$ in the notation of \cref{subsec:rate-drift}).
\begin{enumerate}
    \item Suppose that $k_2z=\Omega(\log^2 n)$. Then for any distinct $v,v'\in V_2$ and any $2\le d\le \log k_2$ we have
\[\Pr[\deg(v)=d]=\frac{z^d}{d!f(z)}\left(1+O\left(\frac{d^2+1}{k_2 z}\right)\right)\]
and
\[\Pr[\deg(v)=\deg(v')=d]=\left(\frac{z^d}{d!f(z)}\right)^2\left(1+O\left(\frac{d^2+1}{k_2 z}\right)\right).\]
\item For any $v\in V_2$ and $d\ge 2$, we have the cruder estimate
\[\Pr[\deg(v)=d]=O\left((k_2z)^{1/2}\frac{z^d}{d!f(z)}\right).\]
\end{enumerate}
\end{lemma}
Note that $z^d/(d!f(z))=\Pr[Q=d\,|\,Q\ge 2]$, for $Q\sim \operatorname{Poisson}(z)$ (i.e., it is a point probability for a truncated Poisson random variable). Very briefly, the proof strategy for \cref{lem:degree-prob} is as follows: one can show that the degree sequence of $\mb G^\ast(n,m)$ is precisely a sequence of independent Poisson random variables conditioned on their sum being exactly $2m$. So, the proof of \cref{lem:degree-prob} essentially comes down to careful estimates on point probabilities of sums of independent truncated Poisson random variables, using standard techniques for proving local limit theorems.

\begin{proof}[Proof of \cref{lem:KS-variance}]
Recalling the fluid limit approximations in \cref{subsec:fluid-limit}, and the definition of $z(x)$ from \cref{subsec:rate-drift}, let $z_\delta=z(\chi(s_\delta))$, and let $Q\sim \on{Poisson}(z_\delta)$. Let $\mu^{(1)}=\chi_1(s_\delta)n$ and $\mu^{(d)}=\chi_2(s_\delta)n\Pr[Q=d\,|\,Q\ge 2]$ for $2\le d\le \log n$, and $\mu^{(d)}=0$ for $d>\log n$. Then
\[\sum_d d\mu^{(d)}\le \left(\chi_1(s_\delta)+\chi_2(s_\delta)\,\mb E[Q\,|\,Q\ge 2]\right)n.\]
By the definition of $z(\chi(x))$, and using \cref{lem:analytic-estimates}(1), we have
\[\mb E[Q\,|\,Q\ge 2]
=O(1),\]
so recalling that $\chi_1(s_\delta),\chi_2(s_\delta)\to 0$ as $\delta\to 0$ (again by \cref{lem:analytic-estimates}(1)), we have $\sum_d d\mu^{(d)}\le \varepsilon n$ for sufficiently small $\delta$.

By \cref{lem:analytic-estimates}(3) and Chebyshev's inequality, together with the Gaussian approximation \cref{eq:gaussian-convergence}, we see that with probability at least $1-\varepsilon/2$ we have 
\begin{equation}\sum_{j\in \{1,2,3\}}|X_j^n(\tau_\delta^n)-\chi_j(s_\delta)n|\le h(\delta) \sqrt n.\label{eq:X-small-fluctuation}\end{equation}
for some $h(\delta)$ tending to zero as $\delta\to 0$. Recalling the definition $D=\sum_d d\,|X^{(d)}-\mu^{(d)}|$ from the statement of \cref{lem:KS-variance}, our goal is now to show that $\mb E[D\,|\,X^n(\tau_\delta^n)]$ is small whenever $X^n(\tau_\delta^n)$ satisfies \cref{eq:X-small-fluctuation} (we will then finish the proof with Markov's inequality). We consider each $|X^{(d)}-\mu^{(d)}|$ separately.

First, note that $X^{(1)}=X_1^n(\tau_\delta^n)$, so when \cref{eq:X-small-fluctuation} holds we have
\begin{equation}
    |X^{(1)}-\mu^{(1)}|\le h(\delta)\sqrt n.\label{eq:X1-fluctuation}
\end{equation}

For the cases where $d\ge 2$ we will apply \cref{lem:degree-prob}. Note that if we consider the remaining graph at time $\tau_\delta^n$, and we condition on its set $V^{(1)}$ of $X_1^n(\tau_\delta^n)=X^{(1)}$ degree-1 vertices, and its set $V^{(\ge 2)}$ of $X_2^n(\tau_\delta^n)$ vertices of degree at least 2, and its total number $X_3^n(\tau_\delta^n)$ of edges, then (up to relabelling vertices), this remaining graph at time $\tau_\delta^n$ is distributed as $\mb G^\ast(X_1^n(\tau_\delta^n)+X_2^n(\tau_\delta^n),X_3^n(\tau_\delta^n))$. This simple fact appears explicitly as \cite[Lemma~2]{AFP98}.

Recalling the definition of $z(x)$ from \cref{subsec:rate-drift}, let $Z^n(s)=z(X^n(s))$. In order to apply \cref{lem:degree-prob}, we first need to show that when \cref{eq:X-small-fluctuation} holds, $Z^n(\tau_\delta^n)$ is well-approximated by its fluid limit approximation $z_\delta$. Indeed, we compute
\begin{equation}\frac{z(e^{z}-1)}{f(z)}=2+z/3+O(z^2).\label{eq:poisson-sensitivity}\end{equation}
For $X^n(\tau_\delta^n)$ satisfying \cref{eq:X-small-fluctuation}, using that $\chi_2(s_\delta)=\Omega(\delta)$ (by \cref{lem:analytic-estimates}(1)) we have
\[\frac{2X^n_3(\tau_\delta^n)-X^n_1(\tau_\delta^n)}{X^n_2(\tau_\delta^n)}=\frac{2\chi_3(s_\delta)-\chi_1(s_\delta)}{\chi_2(s_\delta)}+O\left(\frac{h(\delta)}{\delta\sqrt n}\right),\]
so by \cref{eq:poisson-sensitivity},
\begin{equation}Z^n(\tau_\delta^n)=z(X^n(\tau^n_\delta))=z_\delta+O\left(\frac{h(\delta)}{\delta\sqrt n}\right).\label{eq:Z-approx}\end{equation}

Now, we consider $|X^{(d)}-\mu^{(d)}|$ in the case $2\le d\le \log n$. Condition on outcomes of $(X_j^n(\tau_\delta^n))_{j\in\{1,2,3\}}$ satisfying \cref{eq:X-small-fluctuation}, and also condition on outcomes of the vertex sets $V^{(1)},V^{(\ge 2)}$. In the resulting conditional probability space, let $\mbm 1_v$ be the indicator random variable for the event $\deg(v)=d$. By \cref{lem:degree-prob}, for any distinct $v,v'\in V^{(\ge 2)}$ we have
\begin{align*}
\mb E[\mbm 1_v]&=\frac{Z^n(\tau_\delta^n)^d}{d!f(Z^n(\tau_\delta^n))}\left(1+O\left(\frac{d^2+1}{X_2^n(\tau_\delta^n) Z^n(\tau_\delta^n)}\right)\right)\\
&=\Pr[Q=d\,|\,Q\ge 2]+\frac1{d^{\Omega(d)}}O\left(\frac{h(\delta)}{\delta\sqrt n}\right),\\
\operatorname{Var}[\mbm 1_v]&\le \mb E[\mbm 1_v]\le \frac1{d^{\Omega(d)}}O(1),\\
\operatorname{Cov}[\mbm 1_v,\mbm 1_{v'}]&=\left(\frac{Z^n(\tau_\delta^n)^d}{d!f(Z^n(\tau_\delta^n))}\right)^2 O\left(\frac{d^2+1}{X_2^n(\tau_\delta^n) Z^n(\tau_\delta^n)}\right)\\
&=\frac1{d^{\Omega(d)}}O\left(\frac{1}{\delta\sqrt \delta n}\right).
\end{align*}

For these estimates we have used \cref{eq:Z-approx}, that $\chi_2(s_\delta)=\Theta(\delta)$ and $z_\delta=\Theta(\sqrt \delta)$ (from \cref{lem:analytic-estimates}(1--2)), that $f'(z)=O(1)$ for all $z\in \mb R$, and that $d!=d^{\Omega(d)}$ by Stirling's approximation.

So, for $X^n(\tau_\delta^n)$ satisfying \cref{eq:X-small-fluctuation}, we have
\begin{align*}\mb E[X^{(d)}\,|\,X^n(\tau_\delta^n)]&=X^n_2(\tau_\delta^n)\left(\Pr[Q=d\,|\,Q\ge 2]+\frac1{d^{\Omega(d)}}O\left(\frac{h(\delta)}{\delta\sqrt n}\right)\right)\\
&=\mu^{(d)}+\frac{1}{d^{\Omega(d)}}O(h(\delta)\sqrt n),\\
\on{Var}[X^{(d)}\,|\,X^n(\tau_\delta^n)]&\le X^n_2(\tau_\delta^n)\left(\frac1{d^{\Omega(d)}}O(1)\right)+X^n_2(\tau_\delta^n)^2\left(\frac1{d^{\Omega(d)}}O\left(\frac{1}{\delta\sqrt \delta n}\right)\right)\\
&=\frac{1}{d^{\Omega(d)}}O(\sqrt \delta n).
\end{align*}
 
Using the Cauchy--Schwarz inequality, we deduce
\begin{align}\mb E\left[|X^{(d)}-\mu^{(d)}|\,\middle|\, X^n(\tau_\delta^n)\right]&\le \Big|\mb E[X^{(d)}\,|\,X^n(\tau_\delta^n)]-\mu^{(d)}\Big|+\sqrt{\on{Var}[X^{(d)}\,|\,X^n(\tau_\delta^n)]}\notag\\
&\le \frac{1}{d^{\Omega(d)}}O(h(\delta)\sqrt n + \delta^{1/4} \sqrt n).\label{eq:Xd-fluctuation}
\end{align}

Finally we consider $|X^{(d)}-\mu^{(d)}|$ in the case $d>\log n$. By \cref{lem:degree-prob}(2), again using \cref{lem:analytic-estimates}(1--2),
\begin{align}\mb E\left[|X^{(d)}-\mu^{(d)}|\,\middle|\, X^n(\tau_\delta^n)\right]=\mb E[X^{(d)}\,|\,X^n(\tau_\delta^n)]&\le X_2^n(\tau_\delta^n)\,O\left((X_2^n(\tau_\delta^n)Z^n(\tau_\delta^n))^{1/2}\frac{Z^n(\tau_\delta^n)^d}{d!f(Z^n(\tau_\delta^n))}\right)\notag\\
&\le \frac{1}{d^{\Omega(d)}}O(\delta^{7/4}).\label{eq:Xlogn-fluctuation}
\end{align}

Combining \cref{eq:X1-fluctuation,eq:Xd-fluctuation,eq:Xlogn-fluctuation}, we deduce that whenever $X^n(\tau_\delta^n)$ satisfies \cref{eq:X-small-fluctuation} we have 
\[\mb E[D\,|\,X^n(\tau_\delta^n)]=\sum_dd\,\mb E\left[|X^{(d)}-\mu^{(d)}|\,\middle|\, X^n(\tau_\delta^n)\right]\le O(h(\delta)\sqrt n + \delta^{1/4} \sqrt n)\le (\varepsilon^2/2)\sqrt n\]
for sufficiently small $\delta>0$. So, by Markov's inequality, if we condition on \cref{eq:X-small-fluctuation} we have $D\le \varepsilon \sqrt n$ with probability at least $1-\varepsilon/2$. The desired result follows, recalling that \cref{eq:X-small-fluctuation} holds with probability at least $1-\varepsilon/2$.
\end{proof}

\bibliographystyle{amsplain0.bst}
\bibliography{main.bib}

\end{document}